\documentclass[a4paper,12pt]{amsart} 
\usepackage{amssymb} 
\usepackage{amsmath} 
\usepackage{amsthm}
\usepackage{mathtools}
\usepackage{color}
\usepackage[latin1]{inputenc}
\usepackage{amsfonts}
\usepackage{lipsum}
\usepackage{rotating}
\usepackage{stmaryrd}
\usepackage{accents}
\usepackage[english]{babel}
\usepackage{mathrsfs}
\usepackage{stackrel}
\usepackage[T1]{fontenc}
\usepackage[all]{xy}

         \newtheorem{definition}{Definition}[section]
	\newtheorem{remark}[definition]{Remark}
	\newtheorem{example}{Example}
	\newtheorem{proposition}[definition]{Proposition}
	\newtheorem{almost-proposition}[definition]{Almost-Proposition}

	\newtheorem{theorem}[definition]{Theorem}
	\newtheorem{corollary}[definition]{Corollary}
	\newtheorem{lemma}[definition]{Lemma}
	
	\newtheorem{assumption}[definition]{Assumption}
	\newtheorem*{acknowledgement}{Acknowledgements}

	\def\Hom{{\rm{Hom}}}
	\def\End{{\rm{End}}}
	\def\Aut{{\rm{Aut}}}
	\def\Out{{\rm{Out}}}
	
	\def\dim{{\rm{dim}}}
	
	\def\fin{{\rm{fin}}}
	
	\def\Free{{\rm{Free}}}
	\def\Tilt{{\rm{Tilt}}}
	\def\TrPic{{\rm{TrPic}}}
	\def\StPic{{\rm{StPic}}}

	\def\exp{{\rm{exp}}}
	\def\rad{{\rm{rad}}}

	\def\soc{{\rm{soc}}}	
	\def\id{{\rm{id}}}
	\def\Id{{\rm{Id}}}
	\def\Sph{{\rm{Sph}}}
	\def\Im{{\rm{Im}}}

	\def\Free{{\rm{Free}}}

	\def\St{{\rm{St}}}
	\def\Br{{\rm{Br}}}
	\def\min{{\rm{min}}}
	\def\ker{{\rm{ker}}}
	\def\Im{{\rm{Im}}}

	\def\opp{{\rm{opp}}}

	\def\Re{{\rm{Re}}}

	\def\Sph{{\rm{Sph}}}
	\def\Stab{{\rm{Stab}}}
	\def\CoStab{{\rm{CoStab}}}
	\def\Cone{{\rm{Cone}}}
	\def\Br{{\rm{Br}}}
	\def\P{{\rm{P}}}
	\def\Irr{{\rm{Irr}}}
	
	\def\Ho{{\rm{Ho^{b}}}}

\title{Stability conditions on Brauer tree algebras}
\author{L{\'e}o Dreyfus-Schmidt}
\address{Universti\'e Paris Denis Diderot - Institut de Math\'ematiques de Jussieu - Paris Rive Gauche, Bat\^iment Sophie Germain, 75205 Paris Cedex 13, France }
\email{leo.dreyfus-schmidt@imj-prg.fr}

\begin{document}
\maketitle

\begin{abstract}
We study the space of stability conditions attached to the derived category of $A_{n}\text{-mod}$ for $A_{n}$ the Brauer tree algebra of the line with $n$ edges. In particular we show that for the Brauer tree algebra $A_{3}$, the connected component of the natural heart of the space of stability conditions is simply connected. However, unlike known examples arising in geometry, the Bridgeland homomorphism is not a covering map. 
\end{abstract}

\setcounter{section}{-1}
\section{Introduction}
\subsection{Motivations}
We are here interested in spaces of stability conditions arising in modular representation theory. We study the space of stability conditions on the bounded derived category of the Brauer tree algebra associated to the line. These algebras arise in the study of cyclic defect blocks of group algebras, and they are also related to the \textit{zig-zag algebras} introduced by Huerfano and Khovanov (cf. [HuKh01]).
\\Motivated by Douglas's work on $\pi$-stability for Dirichlet branes, Bridgeland introduced the notion of stability conditions on a triangulated category. Indeed, he shows in [Bri07] that to any triangulated category $\mathcal{D}$, one can associate a complex manifold $\Stab(\mathcal{D})$, parametrising stability conditions on $\mathcal{D}$. This theory has proven particularly interesting in a wide variety of cases, from K3 surfaces to Kleinian singularities ([Bri09]) through bounded derived category of quivers (c.f [Qiu11] and [Ike14]). Closer to our interest, Thomas showed in [Tho06] that a connected component of the space of stability conditions of the perfect derived category of a certain dg-algebra is the universal cover of a configuration space and that the group of deck transformation is an Artin braid group. It appears that the Brauer tree algebras we are interested are un-dg versions of the dg-algebras considered by Thomas and hence this shall build our intuition.\\\\
We will start by introducing some classical material in Section 1 with in particular a quick overview of the theory of spherical objects in Section 1.3 and the theory of elementary perverse equivalences in Section 1.4. Along the way, we will point out an unfaithful action of the affine braid group of type $\tilde{A_{3}}$ on $\mathcal{D}^{b}(A_{\St_{3}})$, the derived category of the Brauer tree algebra $\St_{3}$ associated to the star with three edges and no exceptional vertex. We will make the connection in the case of symmetric algebras between tilted hearts and derived equivalences in Section 1.7. In Section 1.9 and 1.10, we will bring into play Bridgeland's theory of stability conditions.
Then, we will establish in Section 2 a strategy to try to determine the connected component of the natural heart of the Brauer tree algebra associated to the line with $n$ edges and no exceptional vertex. Finally in Section $2.4$, we will show that $\Stab^{0}(A_{3})$ is indeed simply connected and we will reach our most surprising result in Section $2.5$: the Bridgeland local homeomorphism is not a cover of its image. Last but not least, we will end in Section 3 by discussing what happens for the space of co-stability conditions, introduced by J\o rgensen and Pauksztello and see how this does not lead to anything interesting for symmetric algebras. Indeed, we shall see that the space of co-stability conditions of the bounded homotopy category of a symmetric algebras is actually empty.

\begin{acknowledgement}
I would like to thank Rapha\"el Rouquier for suggesting this very interesting problem. Most of this work has been done while visiting City University of London; I am particularly endebted to Joseph Chuang for his precious help and his constant support. Also, I wish to thank Geordie Williamson for many helpful remarks and comments.
\end{acknowledgement}

\subsection{Notations} Let $k$ be a field. We will always consider essentially small $k$-linear categories and we shall assume all subcategories are closed under isomorphisms. The group of autoequivalences of an additive category $\mathcal{C}$ is denoted by $\Aut(\mathcal{C})$. If $\mathcal{S}$ is a subcategory of $\mathcal{C}$, we define its right orthogonal as $\mathcal{S}^{\perp}:=\{M\in \mathcal{C}|\; \Hom(S,M)=0\;, \forall S \in \mathcal{S} \}$. For $\mathcal{A}$ an abelian category, $\mathcal{D}^{b}(\mathcal{A})$ is the bounded derived category of $\mathcal{A}$ while $\Ho(\mathcal{A})$ is its bounded homotopy category. Suppose $\mathcal{C}$ is triangulated with shift functor $[1]$. The distinguished triangles in $\mathcal{C}$ are denoted by diagrams
$$\xymatrix{ a \ar[rr] && b\ar[dl]
\\                   & c \ar@{-->}[ul]}
$$
where the dotted arrow denotes a map $c\rightarrow a[1]$. The categorical operation denoted by $\langle\; \rangle_{\mathcal{C}}$ is the closure under extension in $\mathcal{C}$. For $X$ a set of objects in $\mathcal{C}$, $\langle X \rangle$ is the thick subcategory of $\mathcal{C}$ generated by $X$. Let $A$ be a finite-dimensional $k$-algebra, $A^{\opp}$ will denote the opposite algebra of $A$. The category of finitely generated left $A$-modules (resp. projective) will be denoted by $A\text{-mod}$ (resp.\;$A\text{-proj}$). The corresponding stable module category will be denoted by $A\text{-stmod}$.  By $K_{0}(A)$, we denote the Grothendieck group of finitely generated $A$-modules; the isomorphism classes $[S]$ of the simple $A$-modules form a $\mathbb{Z}$-basis of $K_{0}(A)$. The group of outer automorphisms of $A$ is denoted by $\Out(A)$. The \textit{derived Picard group} $\TrPic(A)$ is the group of isomorphism classes of two-sided tilting complexes for $A\otimes A^{\opp}$ where the product of two classes is given by tensor product. Lastly, we denote by $\StPic(A)$, the \textit{stable Picard group}, defined as the group of projective-free $(A\otimes A^{\opp})$-modules inducing a self-stable equivalence of $A$. We refer to [RoZi03] for a development of the theory of derived Picard groups for derived categories.

\section{Our toolbox}
\subsection{Brauer tree algebras}
In a very pedantic way, we can define a planar tree $\Gamma$ as the data of $(\Gamma_{V},\Gamma_{E},\tau, \omega)$, where:
\begin{itemize}
\item $\Gamma_{V}$ and $\Gamma_{E}$ are finite sets (corresponding respectively to vertices and edges),
\item $\tau$ is a map from $\Gamma_{E}$ to the set of two-element subsets of $\Gamma_{V}$,
\item $\omega$ is the data, for every $s\in \Gamma_{V}$, of a transitive action of $\mathbb{Z}$ on $\Gamma_{E}^{s}$, the set of edges $e$ such that $s\in \tau(e)$ (\textit{i.e.} a cyclic ordering of the edges containing $s$),
\item for any $s,s'$ two vertices, there is a unique sequence $e_{1},\ldots,e_{n}$ of distincts edges such that $s\in \tau(e_{1}), s'\in \tau(e_{n})$ and $\tau(e_{i})\cap \tau(e_{i+1})\neq 0$.
\end{itemize}
A \textit{Brauer tree} is a tree $\Gamma$ together with a specified vertex $v\in \Gamma_{V}$, called the \textit{exceptional} vertex and a positive integer $m$, the \textit{multiplicity} of the exceptional vertex.\\
To a Brauer tree $(\Gamma,v,m)$, one associates a finite-dimensional symmetric $k$-algebra called a \textit{Brauer tree algebra} which is charaterized by the following properties, up to Morita equivalence.
\\The isoclasses of simple modules are parametrized by the edges of the tree. Let the edges of $\Gamma$ be labelled by $1,\ldots, n$ so that $S_{1}\ldots, S_{n}$ are the corresponding representatives of the simple modules. Denote by $P_{i}$ the projective cover of the simple module $S_{i}$. Then $\rad(P_{i})/\soc(P_{i})$ is the direct sum of two uniserial modules, associated with the vertices adjacent to the edge labelled by $i$. For such an adjacent vertex, let the cyclic ordering of the edges be $(i=i_{0},i_{1},\ldots,i_{l},i_{0})$. Then the composition factors of the corresponding uniserial module are, from top to socle, $S_{i_{1}},\ldots,S_{i_{l}},S_{i_{0}},S_{i_{1}},\ldots,S_{i_{l}}$ where the number of composition factors is $m(l+1)-1$ if the vertex labelled by $i$ is the exceptional vertex, and $l$ otherwise. When $m=1$, the choice of an exceptional vertex is irrelevant.\\
For the representation theorist, Brauer tree algebras are of great interest as any block with cyclic defect group of a group algebra is of this type (see [Ale86]). Moreover by a theorem of Rickard (cf. [Ric89]), two Brauer tree algebras with same number of edges and same multiplicity are derived equivalent. In particular, any Brauer tree algebra $(\Gamma,v,m)$ is derived equivalent to the Brauer tree algebra associated to the line with an exceptional vertex at an end having multiplicity $m$.\\
In what follows, we will only be considering Brauer trees where the multiplicity $m$ is 1, \textit{i.e.} there is no exceptional vertex and to a tree $\Gamma$, we will denote by $A_{\Gamma}$ its associated Brauer tree algebra. From now on, we will write $T_{n}$ for the Brauer tree of the line with $n$ edges and by $A_{n}$, its corresponding Brauer tree algebra (we will omit the index $n$ when there is no possible confusion) such that by denoting its simple modules $S_{i}$ for $i=1,\ldots, n$, we have:
$$T_{n}: \xymatrix{\bullet \ar@{-}[r]^{S_{1}} & \bullet \ar@{-}[r]^{S_{2}}& \bullet \cdots \bullet \ar@{-}[r]^{S_{n}}& \bullet}$$ 
Also in what follows, $\St_{n}$ will denote the Brauer tree of the star with $n$ edges.

\subsection{Self-derived equivalences of the line}
We recall here the main results of Rouquier-Zimmermann ([RoZi03]) on self-derived equivalences of $A_{n}$ the Brauer tree algebra of the line. From what was said above, the Loewy series of the projective indecomposable $A_{n}$-modules are as follows: $P_{1}=$ $ \left(\begin{matrix}
   S_{1}  \\
   S_{2}\\
   S_{1}
\end{matrix}\right)$, $P_{n}=$ $ \left(\begin{matrix}
   S_{n-1}  \\
   S_{n-2}\\
   S_{n-1}
\end{matrix}\right)$ and
$P_{i}= $ $ \left(\begin{matrix}
   & S_{i} &  \\
   S_{i-1} & &S_{i+1}\\
   & S_{i} &
\end{matrix}\right)$  for $i\neq 1, n$.\\
Let $f:  \stackrel[i=1]{n}{\bigoplus} P_{i}\otimes P_{i}^{*} \rightarrow A_{n}$ be a projective cover of the $(A_{n},A_{n})$-bimodule $A_{n}$. 
For $i=1\ldots n$, consider the two-term complex of $(A_{n},A_{n})$-bimodules $$X_{i}:=(0\rightarrow P_{i}\otimes P_{i}^{*}\rightarrow A_{n}\rightarrow 0),$$ where $A_{n}$ is in degree $0$ and the differential is given by $f$. This yields a self-derived equivalence $\phi_{i}:\mathcal{D}^{b}(A_{n}) \stackrel{\sim}{\rightarrow}  \mathcal{D}^{b}(A_{n})$. We have $\phi_{i}(S_{j})=S_{j}$ for $j\neq i$ and $\phi_{i}(S_{i})=\Omega(S_{i})[1]$ where  the Heller translate $\Omega(S_{i})$ is the kernel of the projective cover of $S_{i}$. \\We denote by $\Sph(A_{n})$ the subgroup of $\Aut(\mathcal{D}^{b}(A_{n}))$ generated by $\phi_{1},\ldots, \phi_{n}$. The following two results will be of great importance to us. We implicitly identify $\Sph(A_{n})$ with its image in $\TrPic(A_{n})$.

\begin{proposition}([RoZi03, Proposition 4.4])
We have $\Sph(A_{n}) \cap \Out(A_{n})=1.$
\end{proposition}
In [RoZi03], the authors show that the $\phi_{i}$'s define a weak braid group action on $\mathcal{D}^{b}(A_{n})$. We recall the classical presentation of the Artin braid group on $n+1$ strings: $$\Br_{n+1}=\{s_{1}, \ldots,s_{n} \vert \; s_{i}s_{i+1}s_{i}=s_{i+1}s_{i}s_{i+1}, \; s_{i}s_{j}=s_{j}s_{i} \; \text{for}\; |i-j|> 1 \}.$$
\begin{theorem}([RoZi03, Theorem 4.5])
There is a surjective group morphism $\Br_{n+1}\rightarrow \Sph(A_{n})$, $s_{i}\mapsto \phi_{i}$. Also, if $\omega_{0}$ is the longest element of $\mathfrak{S}_{n}$, we have $\omega_{0}^{2}\mapsto[2n]$. 
\end{theorem}
In the case of $A_{2}$, the Brauer tree algebra with only two edges and no exceptional vertex, Rouquier and Zimmermann showed that the previous map is also injective and hence an isomorphism. More generally for any $n\geq 2$, the faithfulness of the braid group action was proved by Khovanov and Seidel in [KhSe02] using geometric techniques.

\begin{theorem}([KhSe02])
The group morphism $\Br_{n+1}\rightarrow \Aut(\mathcal{D}^{b}(A_{n}))$, $\sigma_{i}\mapsto \phi_{i}$ is injective.
\end{theorem}

\subsection{Auto-equivalences and braid groups}
\subsubsection{Spherical objects and twist functors}
We quickly review the definitions and main properties of the theory of spherical objects introduced in [SeTh01]. We will denote by $\mathcal{D}$ the bounded derived category of an abelian category or of a dg-category. One important feature distinguishing $\mathcal{D}$ from arbitrary triangulated categories is that cones are functorial.

\begin{definition}
Let $E\in \mathcal{D}$ a bounded complex and suppose that $\Hom^{\bullet}(-,E)$ and $\Hom^{\bullet}(E,-)$ have finite total dimension. Let $n\in \mathbb{N}^{*}$, we say that $E$ is $n$-spherical if
\begin{enumerate}
\item $\Hom^{i}(E,E)$ is equal to $k$ for $i=0$ and $i=n$, and zero otherwise.
\item The pairing $\Hom^{j}(F,E)\times \Hom^{n-j}(E,F)\rightarrow \Hom^{n}(E,E)\simeq k$ is nondegenerate for all $F\in \mathcal{D}$ and $j\in \mathbb{Z}$.
\end{enumerate}
To a spherical object, we associate a \textit{twist} functor $\phi_{E}:\mathcal{D}\rightarrow \mathcal{D}$ by $\phi_{E}(F)=\Cone(\Hom(E,F)\otimes E \rightarrow F)$.
\end{definition}

However the case $n=0$ has to be treated separately. A $0$-spherical object $E$ is such that $\Hom^{\bullet}(E,E)$ is two-dimensional and concentrated in degree zero and the pairings $\Hom^{j}(F,E)\times \Hom^{-j}(E,F)\rightarrow \Hom^{0}(E,E)/ k\cdot\id_{E}$ are nondegenerate.

\begin{definition}
An $(A_{m})$-configuration in $\mathcal{D}$ is a collection of $n$-spherical objects $(E_{1},\cdots,E_{m})$ satisfying
$$\dim_{k}\Hom_{\mathcal{D}}^{\bullet}(E_{i},E_{j})= \begin{cases}1 \; \vert i-j \vert=1,
  \\ 0\; \vert i-j\vert \geq 2.
  \end{cases}$$
\end{definition}

\begin{theorem}([SeiTho])
The twist $\phi_{E}$ along a spherical object $E$ is an exact auto-equivalence of $\mathcal{D}$. Moreover if $(E_{1},\cdots,E_{m})$ is an $(A_{m})$-configuration, we have
$\Br_{m+1}\rightarrow \Aut(\mathcal{D})$, $\sigma_{i}\mapsto \phi_{E_{i}}$.
\end{theorem}

The proof of this theorem relies on the following proposition which allows one to extend this result to more general types of configurations. Indeed for $\Gamma$ a tree, we can define a $\Gamma$-configuration as a collection of spherical objects $(E_{i})_{i\in \Gamma_{E}}$ in $\mathcal{D}$, such that if $i$ and $j$ share a vertex then $\dim_{k}\Hom_{\mathcal{D}}^{\bullet}(E_{i},E_{j})=1$ and is $0$ otherwise.

 \begin{proposition}
Let $E, F\in \mathcal{D}$ be two $n$-spherical objects.\\
If the total dimension of $\Hom^{\bullet}(E,F)$ is zero, then $\phi_{E} \phi_{F}\simeq \phi_{F} \phi_{E}$.\\
If the total dimension of $\Hom^{\bullet}(E,F)$ is one, $\phi_{E} \phi_{F} \phi_{E}\simeq \phi_{F} \phi_{E} \phi_{F}$.
\end{proposition}

Let $A_{\Gamma}$ be Brauer tree algebra associated to $\Gamma$ a Brauer tree with no exceptional vertex,  and $\{P_{i},\; i\in \Gamma_{V}\}$, the set of isoclasses of projective indecomposables ${A_{\Gamma}}$-modules. Then the $\{P_{i},\; i\in \Gamma_{V}\}$ are $0$-spherical objects and they form a $\Gamma$-configuration. We denote by $\Sph(A_{\Gamma})$ the group generated by their twist functors $\{\phi_{P_{i}},\; i\in \Gamma_{V}\}$. Let $\Br_{\Gamma}$ be the braid group of type $\Gamma^{*}$, the dual graph of $\Gamma$.

\begin{corollary}
We have a surjective group homomorphism $\Br_{\Gamma}\rightarrow \Sph(A_{\Gamma})$.
\end{corollary}

\subsubsection {Differential graded Brauer tree algebra}
We introduce here a dg-version of the Brauer tree algebra of the line $A_{n}$. Let $Q_{n}$ be the following graded quiver with vertices numbered $1,\cdots, n$ and where each edge is of degree 1: 
$$
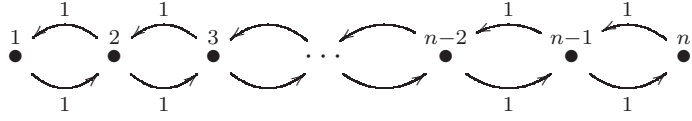
\xymatrix{
\stackrel{1}{\bullet} \ar@/_1pc/[r]_{1}& \stackrel{2}{\bullet} \ar@/_1pc/[l]_{1} \ar@/_1pc/[r]_{1}& \stackrel{3}{\bullet} \ar@/_1pc/[l]_{1} \ar@/_1pc/[r]& \cdots \ar@/_1pc/[l] \ar@/_1pc/[r] &  \stackrel{n-2}{\bullet} \ar@/_1pc/[r]_{1}\ar@/_1pc/[l]& \stackrel{n-1}{\bullet} \ar@/_1pc/[l]_{1} \ar@/_1pc/[r]_{1}& \stackrel{n}{\bullet} \ar@/_1pc/[l]_{1} }
 $$
Let $J$ be the two-sided ideal of the path algebra $kQ_{n}$ generated by the elements $(i\vert i-1\vert i)-(i\vert i+1\vert i)$, $(i-1\vert i \vert i+1)$ and $(i+1\vert i\vert i-1)$ for all $i=2,\cdots, n-1$. Denote the lazy path at the $i$-th edge by $e_{i}$ so that $1=\sum_{i}e_{i}$. The quotient $kQ_{n}/J$ is again a graded algebra and we denote by $\widehat{A_{n}}$ the corresponding dg-algebra with zero differential.\\
Let $\mathcal{D}(\widehat{A_{n}})$ be the bounded derived category of differential graded modules over $\widehat{A_{n}}$ and denote by $\mathcal{D}_{n}$ the perfect dg-derived category of $\widehat{A_{n}}$, \textit{i.e.} the subcategory of $\mathcal{D}(\widehat{A_{n}})$ generated by the projective modules $P_{i}:=e_{i}\widehat{A_{n}}$. Using Proposition $1.7$, we can deduce the dg-analog of Rouquier-Zimmermann's Theorem $1.2$. Alternatively, this can be found in [SeiTh01].
\begin{theorem}
The projectives $P_{i}$ form an $(A_{n})$-configuration of $2$-spherical objects. Hence, we have the group morphism $\Br_{n+1}\rightarrow \Aut(\mathcal{D}_{n})$, $\sigma_{i}\mapsto \phi_{P_{i}}$.
\end{theorem}

\begin{remark}
Forgetting the differential and the grading in $\widehat{A_{n}}$, we recover $A_{n}$ the Brauer tree algebra of the line. In particular, the $A_{n}$-modules $P_{i}$ are 0-spherical objects and the twists $\phi_{P_{i}}$ along $P_{i}$ are actually what we previously denoted by $\phi_{i}$.
\end{remark}

In that sense, it is interesting to keep in mind what happens for the space of stability conditions of $\widehat{A_{n}}$ for our study (cf. [Tho06]).

\subsection{Elementary perverse equivalences}
\subsubsection{Definitions and first properties}
We will now present a very nice family of derived equivalences, the so-called \textit{elementary perverse equivalences} initially introduced by Okuyama and Rickard (cf. [Oku98]). The name comes from Chuang and Rouquier's more general theory of perverse equivalences (cf. [ChRo]). Let us recall the definition and the first properties of elementary perverse equivalences for $A$ a symmetric finite-dimensional basic $k$-algebra. We denote by $\Irr(A)$ the set of isomorphism classes of simple $A$-modules. Let $I\subset \Irr(A)$ and $M\in {A}\text{-mod}$. We denote by $M\rightarrow I_{M}$ an injective hull and $\phi_{M}:P_{M}\rightarrow M$ a projective cover of $M$. We denote by $M_{I}$ the largest quotient of $P_{M}$ by a submodule of $\ker\: \phi_{M}$ such that all composition factors of the kernel of the induced map $M_{I}\rightarrow M$ are in I.\\
Finally, $M^{I}$ will denote the largest submodule of $I_{M}$ containing $M$ and such that all composition factors of $M^{I}/M$ are in $I$.\\
Following constructions of Rickard and Okuyama, we can construct a tilting complex for $A$ as follows. For $S\in I$, we let $Q_{S}$ be a projective cover of the kernel of the canonical map $P_{S}\rightarrow S_{I}$. This gives us a map $Q_{S}\stackrel{f_{S}}{\rightarrow} P_{S}$ and we can define the following complex
$$T_{A,S}(I)=0\rightarrow Q_{S}\stackrel{f_{S}}{\rightarrow} P_{S}\rightarrow 0
$$ where $P_{S}$ is in degree 0.
For $S \in \Irr(A)-I$, we put
$$
T_{A,S}(I)=0\rightarrow P_{S}\rightarrow 0
$$ where $P_{S}$ is in degree $-1$.\\
Now we claim that $T_{A}(I)=\bigoplus_{S\in \Irr(\mathcal{A})}T_{A,S}(I)$ is a tilting complex. We let $A'=\End_{\mathcal{D}^{b}({A})}(T_{A}(I))$.

\begin{proposition}
The tilting complex $T_{A}(I)$ induces a derived equivalence $L_{A,I}:\mathcal{D}^{b}({A})\stackrel{\sim}{\rightarrow} \mathcal{D}^{b}({A'})$. 
\end{proposition}

Since the non-isomorphic indecomposable summands of the tilting complex $T_{A}(I)$ are indexed by $\Irr(A)$, this equivalence induces a bijection $\Irr(A)\stackrel{\sim}{\rightarrow} \Irr(A'), S\mapsto S'$. We say that $L_{A,I}$ and $L_{A,I}^{-1}$ are \textit{elementary perverse equivalences}. Also, we will write $L_{I}$ for $L_{A,I}$ when there is no possible confusion. The next result will be of great importance in what follows; it determines the image of the simple modules under an elementary perverse equivalence.

\begin{lemma} Let $S' \in  \Irr(A')$. We have
$$L_{A,I}^{-1}(S')[-1]=\begin{cases}
  S[-1] \;\text{for}\: S\in I.
  \\ S^{I} \;\text{for}\: S\notin I.
  \end{cases}$$
\end{lemma}

\begin{remark}
We can extend the previous construction to $I=\emptyset$ and $I=\Irr(A)$ as follows. For $I=\emptyset$, we have $L_{A,\emptyset}(A\text{-mod})=(A\text{-mod})[-1]$ while for $I=\Irr(A)$, we have that $L_{A,\Irr(A)}(A\text{-mod})=A\text{-mod}$. 
\end{remark}

\subsubsection{Elementary perverses equivalences on Brauer tree algebras}
\begin{proposition}
Let $A_{n}$ be the basic Brauer tree algebra associated to a line with n edges and no exceptional vertex and $\phi_{i}\in \Aut(\mathcal{D}^{b}(A_{n}))$ the self-derived equivalence previously defined. Then $\phi_{i}^{-1}(A_{n}\text{-mod})=L_{A,\llbracket 1,n \rrbracket-\{i\}}(A_{n}\text{-mod})$.
\end{proposition}

\begin{proof}
This is a direct application of the previous lemma.
\end{proof}

 In what follows, we will write $F_{i}$ for $\phi_{i}^{-1}$. We now try to understand how an elementary perverse equivalences affects the tree attached to a Brauer tree algebra.
\\Let $\Gamma$ be a Brauer tree and $A_{\Gamma}$ the associated Brauer tree algebra. For $I\subseteq \Irr(A_{\Gamma})$, we will now explain how the Brauer tree $\Gamma$ changes upon applying the elementary perverse equivalence $L_{A_{\Gamma},I}$, \textit{i.e.} we determine the Brauer tree associated to the algebra $\End_{A_{\Gamma}}(T_{A_{\Gamma}}(I))$. This can be found in [ChRo] or in [Aih10] in the case of simple tilts, $|I|=1$.\\
We define a Brauer tree $\Gamma'$ as follows. Its set of vertices $\Gamma'_{V}$ is in bijection with $\Gamma_{V}$, $\Gamma_{V}\stackrel{\sim}{\rightarrow} \Gamma'_{V}, x\mapsto x'$. As for the edges, we have a bijection $\phi:\Gamma_{E}\stackrel{\sim}{\rightarrow} \Gamma'_{E}$ given as follows. Let $\{x,y\}\in \Gamma_{E}$ an edge of $\Gamma$. If $\{x,y\}\not\in I$, $\phi(\{x,y\})=\{x',y'\}$. Assume now $\{x,y\}\in I$ and we put $\phi(\{x,y\})=\{a,b\}$ where $a$ and $b$ are described as follows.\\
$\bullet$ If there is an edge with vertex $x$ that is not contained in $I$: let $y=y_{0},\ldots,y_{r}$ be the ordered vertices adjacent to $x$, and let $k=\min\{i\in \llbracket 0,r\rrbracket\;| \; \{x,y_{i}\} \not\in I\}$. We put $a=y_{k}'$. If all edges around $x$ are in $I$, then $a=x'$.\\
$\bullet$ If there is an edge with vertex $y$ that is not contained in $I$, we proceed as above. Let $x=x_{0},\ldots,x_{s}$ be the ordered vertices adjacents to $y$, then let $l=\min\{i\in \llbracket 0,s\rrbracket\;| \; \{x,y_{i}\} \not\in I\}$. We put $b=x_{l}'$. If all edges around $y$ are in $I$, then $b=y'$.\\
Last but not least, we need to describe the cyclic ordering of the vertices adjacent to a vertex $x'$.\\
$\bullet$ If all edges containing $x$ are in $I$ and if $y_{0},\ldots, y_{r}$ are the ordered vertices around $x$, then the ordered vertices around $x'$ are $y'_{0},\ldots, y'_{r}$.\\
$\bullet$ Assume there is a vertex $y_{0}$ adjacent to $x$ such that $\{x,y_{0}\}\not \in I$. Now to give a cyclic ordering of the vertices adjacent to $x$ amounts to the action of an automorphism $l_{x}$. Given an edge $\{x,y\}$, the ordered $N_{x}$ vertices adjacent to $x$ are $y,l_{x}(y),\ldots, l_{x}^{N_{x}}(y)$.  Let $y_{0},\ldots, y_{t}$ be the ordered sequences of vertices adjacent to $x$ such that for all $i\in \llbracket 0,t\rrbracket$, $\{x,y_{i}\}\not \in I$. Let $r_{i}\geq 0$ be the minimal integer such that $\{y_{i},l_{y_{i}}^{-r_{i}-1}(x)\} \not\in I$. Then, the ordered vertices around $x'$ are $y'_{0},l_{y_{0}}^{-1}(x)',\ldots, l_{y_{0}}^{-r_{0}}(x)',y'_{1},l_{y_{1}}^{-1}(x)',\ldots, l_{y_{t}}^{-r_{t}}(x') $.\\
The following proposition is proved in [ChRo] for general tilts and in [Aih10] for simple tilts.

\begin{proposition}
The algebra $\End_{A_{\Gamma}}(T_{A_{\Gamma}}(I))$ is a Brauer tree algebra with tree $\Gamma'$.
\end{proposition}

We will write $\Gamma'=\mathcal{L}_{I}\Gamma$. This notation will become natural after Proposition $1.27$.

\begin{example}
We have the following table of Brauer trees:\\
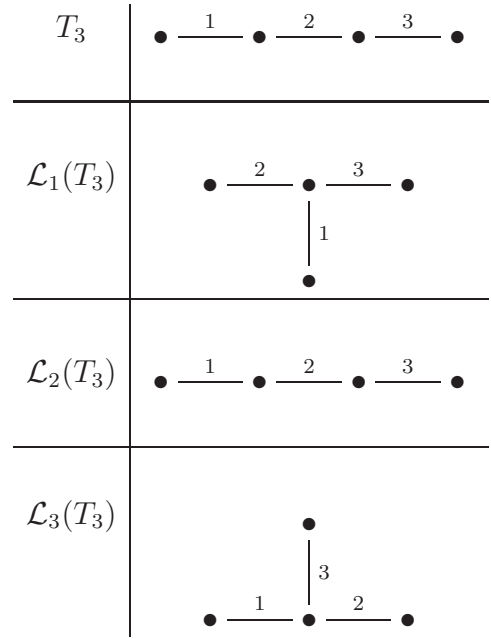
\begin{center}
    \begin{tabular}{ c | c}
   
   $T_{3}$ & \xymatrix{\bullet \ar@{-}[r]^{{1}} & \bullet \ar@{-}[r]^{{2}}& \bullet \ar@{-}[r]^{{3}}& \bullet} \\[20pt] \hline
    \\[5pt ]$\mathcal{L}_{1}(T_{3})$ & \xymatrix{\bullet \ar@{-}[r]^{{2}} & \bullet \ar@{-}[r]^{{3}}& \bullet\\ & \bullet \ar@{-}[u]_{{1}} }  \\[20pt] \hline
   \\[5pt ] $\mathcal{L}_{2}(T_{3})$ & \xymatrix{\bullet \ar@{-}[r]^{{1}} & \bullet \ar@{-}[r]^{{2}}& \bullet \ar@{-}[r]^{{3}}& \bullet} \\[20pt] \hline
    \\[5pt ]$\mathcal{L}_{3}(T_{3})$ & \xymatrix{ & \bullet \ar@{-}[d]^{{3}} \\ \bullet \ar@{-}[r]^{{1}} & \bullet \ar@{-}[r]^{{2}}& \bullet} 
    \end{tabular}
\end{center}
\end{example}

\begin{proposition}
Let $(\Gamma,v,m)$ and $(\Gamma',v,m)$ be two Brauer trees with $n$ edges. Then there exists a derived equivalence $F:\mathcal{D}^{b}(A_{\Gamma})\stackrel{\sim}{\rightarrow} \mathcal{D}^{b}(A_{\Gamma'})$ such that $F$ is the composition of elementary perverse equivalences.
\end{proposition}

\begin{proof}
In [Ric89], Rickard proved that any Brauer tree algebra is derived equivalent to the star-tree algebra with same numerical invariant and where all the edges are adjacent to the exceptional vertex. The derived equivalence that is constructed decomposes into a product of elementary perverse equivalences according to [ChRo]. That way, if $\St_{n} $ denotes the star tree with $n$ edges, we have $\mathcal{D}^{b}(A_{\Gamma})\stackrel{\sim}{\rightarrow} \mathcal{D}^{b}(A_{\St_{n}}) \stackrel{\sim}{\leftarrow}  \mathcal{D}^{b}(A_{\Gamma'})$. 
\end{proof}


\subsection{$t$-structures and tilting}
\subsubsection{$t$-structures and hearts}
We introduce here the notion of $t$-structure on a triangulated category. 

\begin{definition}
A $t$-structure on a triangulated category $\mathcal{C}$ is the data for $i\in \mathbb{Z}$ of full subcategories $\mathcal{T}^{\leq i}$ and $\mathcal{T}^{\geq i}$ with
\begin{itemize}
\item $\mathcal{T}^{\leq i+1}[1]=\mathcal{T}^{\leq i}$ and $\mathcal{T}^{\geq i+1}[1]=\mathcal{T}^{\geq i}$,
\item $\mathcal{T}^{\leq 0} \subseteq \mathcal{T}^{\leq 1}$ and $\mathcal{T}^{\geq 1} \subseteq \mathcal{T}^{\geq 0}$,
\item $\Hom(\mathcal{T}^{\leq 0}, \mathcal{T}^{\geq 1})=0$,
\item given $X\in \mathcal{T}$, there is a distinguished triangle $Y\rightarrow X\rightarrow Z\leadsto$ with $Y \in \mathcal{T}^{\leq 0}$ and $Z \in \mathcal{T}^{\geq 1}$.
\end{itemize}
Moreover the heart of a $t$-structure is defined as the intersection $\mathcal{T}^{\leq 0}\cap \mathcal{T}^{\geq 0}$, and this is an abelian category (cf. [BDD82]). 
\end{definition}

We will refer to $\mathcal{T}^{\leq 0}$ (resp. $\mathcal{T}^{\geq 0}$) as the \textit{left aisle} (resp. \textit{right aisle}) of the $t$-structure. The standard $t$-structure on the bounded derived category of an abelian category $\mathcal{A}$, $\mathcal{D}^{b}(\mathcal{A})$ consists of the left aisle $\mathcal{T}^{\leq0}=\{X\in \mathcal{D}^{b}(\mathcal{A})\:|\: H^{i}(X)=0\;\text{for}\; i> 0  \}$ and the right aisle $\mathcal{T}^{\geq0}=\{X\in \mathcal{D}^{b}(\mathcal{A})\:|\: H^{i}(X)=0\;\text{for}\; i< 0  \}$. The corresponding standard heart is then $\mathcal{A}$.

\subsubsection{Torsion theories and tilting}
From now on, $\mathcal{C}$ denotes an essentially small triangulated category that is, $\mathcal{C}$ is equivalent to a category in which the class of objects is a set.
\begin{definition}
A \textit{torsion theory} in an abelian category $\mathcal{A}$ is a subcategory $\mathcal{T}$ such that every $a\in \mathcal{A}$ fits into a short exact sequence $$0\rightarrow t\rightarrow a \rightarrow f\rightarrow 0$$
 for some $t \in \mathcal{T}$ and $f\in \mathcal{T}^{\perp}:=\mathcal{F}$. We refer to the objects of $\mathcal{T}$ as torsion objects while objects of $\mathcal{F}$ are known as torsion-free objects.
\end{definition}

This terminology is directly inspired by the subcategories of torsion and torsion-free abelian groups inside the category of abelian groups. Assume now that $\mathcal{A}$ is a finite length category, that is $\mathcal{A}$ is both artinian and noetherian so that every object in $\mathcal{A}$ has a finite composition series.
Then we have that torsion theories in $\mathcal{A}$ form a poset for inclusion. For instance, each simple object $s\in \mathcal{A}$ determines two torsion theories, $\langle s\rangle_{\mathcal{C}}$ and $^{\perp}\langle s\rangle_{\mathcal{C}}$ which are respectively minimal and maximal for the partial order. In the following, we will say that $\langle s\rangle_{\mathcal{C}}$ is the torsion theory generated by the simple $s$.\\
Works of Happel, Reiten, Smal\o \;\;([HaReSm96]) and Beligiannis and Reiten ([BeRe07]) lead to the following proposition which makes the connection between torsion theories and $t$-structures. As a $t$-structure is entirely determined by its left (or right) aisle, we will here allow ourselves to identify a $t$-structure with its left aisle.

\begin{proposition}
Let $\mathcal{D}$ be the left aisle of a $t$-structure in $\mathcal{C}$ with heart $\mathcal{H}$. There is an isomorphism between the poset of torsion theories in $\mathcal{H}$ and the interval in the poset of t-structures consisting of $t$-structures with left aisle $\mathcal{E}$ such that $\mathcal{D}\subset \mathcal{E} \subset \mathcal{D}[-1]$.
\end{proposition}

In this way a torsion theory $\mathcal{T}$ in an abelian category $\mathcal{H}$ determines a new heart in $\mathcal{C}$, denoted by $\mathcal{L}_{\mathcal{T}}\mathcal{H}$. We say this heart is obtained from \textit{left tilting at} $\mathcal{T}$. This \textit{tilted} heart can be explicitly written as :
$$\mathcal{L}_{\mathcal{T}}\mathcal{H}=\{c \in \mathcal{C}|\: H^{0}(c) \in \mathcal{F},  H^{1}(c) \in \mathcal{T}\: \text{and}\: H^{i}(c)=0\: \text{for} \: i\neq 0,1  \}.$$
In other words, $\mathcal{L}_{\mathcal{T}}\mathcal{H}=\langle \mathcal{F}, \mathcal{T}[-1] \rangle_{\mathcal{C}}$ where $\langle - \rangle_{\mathcal{C}}$ denotes the closure under extension in $\mathcal{C}$. In an analogous way, a torsion theory also determines another $t$-structure $\langle \mathcal{D}^{\perp}, \mathcal{F}\rangle_{\mathcal{C}}^{\perp}$ whose heart $\mathcal{R}_{\mathcal{T}}\mathcal{H}$ is the right tilt at $\mathcal{T}$ .\\ 
Let us note here that $\mathcal{R}_{\mathcal{T}}\mathcal{H}=\mathcal{L}_{\mathcal{T}}\mathcal{H}[1]$. Also $\mathcal{F}$ is a torsion theory in $\mathcal{L}_{\mathcal{T}}\mathcal{H}$ and $\mathcal{R}_{\mathcal{F}}\mathcal{L}_{\mathcal{T}}\mathcal{H}=\mathcal{H}$. For $s$ a simple object in $\mathcal{H}$, we will write $\mathcal{L}_{s}=\mathcal{L}_{\langle s \rangle_{\mathcal{C}}}$ and $\mathcal{R}_{s}=\mathcal{R}_{^{\perp}\langle s \rangle_{\mathcal{C}} }$ and we shall refer to those as the \textit{simple tilts} of $\mathcal{H}$. We now give a description of the simple objects in the tilted heart $\mathcal{L}_{\mathcal{T}}\mathcal{H}$. Also, we will write $\mathcal{L}^{+}_{\mathcal{T}}\mathcal{H}=\mathcal{L}_{\mathcal{T}}\mathcal{H}$ and $\mathcal{L}^{-}_{\mathcal{T}}\mathcal{H}=\mathcal{R}_{\mathcal{T}}\mathcal{H}$.

\begin{lemma}{[Wo10, Lemma 2.4]}
Let $\mathcal{T}$ be a torsion theory in $\mathcal{H}$. Then any simple object in $\mathcal{L}_{\mathcal{T}}\mathcal{H}$ lies either in $\mathcal{F}$ or in  $\mathcal{T}[-1]$. Moreover,

\begin{enumerate}
\item $f\in \mathcal{F}$ is simple in $\mathcal{L}_{\mathcal{T}}\mathcal{H}$ if and only if there are no exact triangles
$$f'\rightarrow f \rightarrow f''\rightarrow f'[1]\:\text{or}\; t'[-1]\rightarrow f \rightarrow f'\rightarrow t'
$$ where $f',f''\in \mathcal{F}$ and $t'\in \mathcal{T}$ are all non-zero.
\item $t[-1] \in \mathcal{T}[-1]$ is simple in $\mathcal{L}_{\mathcal{T}}\mathcal{H}$ if and only if there are no exact triangles
$$t'\rightarrow t \rightarrow t''\rightarrow t'[1]\:\text{or}\; t'[-1]\rightarrow t[-1] \rightarrow f'\rightarrow t'
$$ where $t',t''\in \mathcal{T}$ and $f'\in \mathcal{F}$ are all non-zero.

\end{enumerate}
\end{lemma}
Moreover for a simple torsion theory, following [Wo10] and using the lemma, we can explicitely determine the simple $\mathcal{L}_{\mathcal{T}}\mathcal{H}$-modules. We now recall two very useful results that can also be found in [Wo10].

\begin{lemma}{[Wo10, Lemma 2.7]}
Any left tilt at a torsion theory containing finitely many isoclasses of indecomposables is a composition of simple left tilts.
\end{lemma}

\begin{lemma}
Let $\Phi: \mathcal{C}\rightarrow \mathcal{C}$ a triangulated equivalence, $\mathcal{H}$ a heart of $\mathcal{C}$ and $s$ a simple object in $\mathcal{H}$. Then $\Phi(\mathcal{L}_{s}\mathcal{H})=\mathcal{L}_{\Phi(s)}\Phi(\mathcal{H})$.
In other words, the actions of automorphisms of $\mathcal{C}$ commute with simple tilts. 
\end{lemma}

\subsubsection{Ordered hearts and tilting}
We now introduce ordered hearts, \textit{i.e.} hearts together with a labelling of their simples. Assume $\mathcal{H}$ has a finite number of isoclasses of simples objects and consider a fixed labelling of its simples $\phi:\llbracket 1,n \rrbracket \stackrel{\sim}{\rightarrow} \Irr(\mathcal{H})$. This ordered heart will be denoted by $\mathcal{H}^{\phi}$. That way for $I\subset \Irr(\mathcal{H})$ we will write $\mathcal{L}_{I}\mathcal{H^{\phi}}=\mathcal{L}_{\phi^{-1}(I)}\mathcal{H}$ and $\Irr(\mathcal{L}_{I}\mathcal{H})$ inherits the labelling from $\Irr(\mathcal{H})$. Indeed, they both have the same number of simple objects, since the classes of those form a basis for the Grothendieck group of $\mathcal{C}$. 

\begin{example}
We consider $A_{n}$ the Brauer tree algebra associated to the line, so that the $A_{n}$-simple are labelled by $\phi:\llbracket 1,n \rrbracket \stackrel{\sim}{\rightarrow} \Irr(A_{n}\text{-mod})$ as follows $\xymatrix{\bullet \ar@{-}[r]^{S_{1}} & \bullet \ar@{-}[r]^{S_{2}}& \bullet \cdots \bullet \ar@{-}[r]^{S_{n}}& \bullet}$. The abelian category $A_{n}\text{-mod}$ is the canonical heart of $\mathcal{D}^{b}(A_{n})$. We can then left tilt $A_{n}\text{-mod}$  at $S_{1}$ and the new simple objects in $\mathcal{L}_{S_{1}}(A_{n}\text{-mod})$ are $T_{1}=S_{1}[-1],  T_{2}=V,$ the unique (up to isomorphism) nonsplit extension of $S_{1}$ by $S_{2}$ and $T_{3}=S_{3},\ldots, T_{n}=S_{n}$.\\
That way, we can write $\mathcal{L}_{S_{2}}\mathcal{L}_{V}\mathcal{L}_{S_{1}}(A_{n}\text{-mod})=\mathcal{L}_{1}\mathcal{L}_{2}\mathcal{L}_{1}({A_{n}\text{-mod}}^{\phi})$. Moreover after a straightforward computation, the following braid relation appears $\mathcal{L}_{1}\mathcal{L}_{2}\mathcal{L}_{1}(A_{n}\text{-mod})=\mathcal{L}_{2}\mathcal{L}_{1}\mathcal{L}_{2}(A_{n}\text{-mod})$.
\end{example}

Let us consider an ordered heart $\mathcal{H}^{\phi}$ with $\phi:\llbracket 1,n \rrbracket \stackrel{\sim}{\rightarrow} \Irr(\mathcal{H})$. We have an action of the symmetric group $\mathfrak{S}_{n}$ on the set of ordered hearts with the same underlying heart, by relabelling the simples. If $\mathcal{T}_{1},\mathcal{T}_{2}, \ldots, \mathcal{T}_{j}$ and $\mathcal{T}'_{1},\mathcal{T}'_{2}, \ldots, \mathcal{T}'_{j'}$ are two sequences of torsion theories generated by simples in $\mathcal{H}$ such that $ \mathcal{L}_{\mathcal{T}_{1}} \mathcal{L}_{\mathcal{T}_{2}}\ldots  \mathcal{L}_{\mathcal{T}_{j}}\mathcal{H}=\mathcal{L}_{\mathcal{T}'_{1}}\mathcal{L}_{\mathcal{T}'_{2}}\ldots \mathcal{L}_{\mathcal{T}'_{j'}} \mathcal{H}$ and $\{T_{1},\ldots,T_{n}\}$ (resp. $\{T'_{1},\ldots,T'_{n}\}$) is the set of ordered simples of $\mathcal{L}_{\mathcal{T}_{1}} \mathcal{L}_{\mathcal{T}_{2}}\ldots  \mathcal{L}_{\mathcal{T}_{j}}\mathcal{H}$ (resp. $\mathcal{L}_{\mathcal{T}'_{1}}\mathcal{L}_{\mathcal{T}'_{2}}\ldots \mathcal{L}_{\mathcal{T}'_{j'}} \mathcal{H}$) by keeping track of the labelling induced by $\phi$. Then there is $\sigma \in \mathfrak{S}_{n}$ such that $T'_{i}=T_{\sigma(i)}$ for all $i \in \llbracket 1,n \rrbracket$ and we write  $ \mathcal{L}_{\mathcal{T}_{1}} \mathcal{L}_{\mathcal{T}_{2}}\ldots  \mathcal{L}_{\mathcal{T}_{j}}\mathcal{H}^{\phi}=\sigma\mathcal{L}_{\mathcal{T}'_{1}}\mathcal{L}_{\mathcal{T}'_{2}}\ldots \mathcal{L}_{\mathcal{T}'_{j'}} \mathcal{H}^{\phi}$.
\\Also for $I\subseteq \llbracket 1,n \rrbracket$ and $\sigma \in \mathfrak{S}_{n}$, we have $\sigma.\mathcal{L}_{I}(\mathcal{H}^{\phi})=\mathcal{L}_{\sigma^{-1}(I)}(\mathcal{H}^{\sigma.\phi})$ for any heart $\mathcal{H}^{\phi}$. We want to point out here that remembering the extra information of the permutation $\sigma$ will turn out to be very important later on.

\begin{example}
Pursuing the previous example, we find that the ordered simples of $\mathcal{L}_{1}\mathcal{L}_{2}\mathcal{L}_{1}(A_{n}\text{-mod}^{\phi})$ are $S_{2}[-1], S_{1}[-1], V, S_{4}, \ldots, S_{n}$ where $V$ is the unique (up to isomorphism) non split extension of $S_{3}$ by $S_{2}$. On the other hand, the ordered simples of $\mathcal{L}_{\{1,2\}}(A_{n}\text{-mod}^{\phi})$ are $S_{1}[-1], S_{2}[-1], V,$ $S_{4}, \ldots, S_{n}$. Hence we have $\mathcal{L}_{1}\mathcal{L}_{2}\mathcal{L}_{1}(A_{n}\text{-mod}^{\phi})=(12)\mathcal{L}_{\{1,2\}}(A_{n}\text{-mod}^{\phi})$.
\end{example}

Consider again the settings of the section 1.4.1. Let $A$ be a finite dimensional symmetric $k$-algebra and $\mathcal{C}=\mathcal{D}^{b}(A\text{-mod})$. For $\phi:\llbracket 1,n \rrbracket \stackrel{\sim}{\rightarrow} \Irr({A}\text{-mod})$ a labelling of the simples of $A\text{-mod}$, write $\Tilt(A\text{-mod}^{\phi})$ for the set of hearts in $\mathcal{C}$ that are obtained from $A\text{-mod}^{\phi}$ by a finite sequence of tilts. Let $\mathcal{P}(\llbracket 1,n \rrbracket)$ denote the set of subsets of $\llbracket 1,n \rrbracket$.\\
We have an action of $\TrPic(A)$ on $\mathcal{C}$ which induces a left action of $\TrPic(A)$ on $\Tilt(A\text{-mod}^{\phi})$. On the other hand, there is a right action of $\Free(\mathcal{P}(\llbracket 1,n \rrbracket))\rtimes \mathfrak{S}_{n}$ on $\Tilt(A\text{-mod}^{\phi})$. The action of $\mathfrak{S}_{n}$ was explained above and is given by permutation of indices, while for $I\in \mathcal{P}(\llbracket 1,n \rrbracket)$ and $\mathcal{H}\in \Tilt(A\text{-mod}^{\phi})$ we have $I.\mathcal{H}=\mathcal{L}_{I}(\mathcal{H})$. Also we have $\emptyset.\mathcal{H}=\mathcal{H}$ and $\llbracket 1,n \rrbracket.\mathcal{H}=\mathcal{H}[-1]$ .
\\Now as a counterpart of Lemma 1.22, those two actions commute. For $\varphi\in \TrPic(A)$, $I\subseteq \llbracket 1,n \rrbracket$ and $\mathcal{H}\in \Tilt(A\text{-mod}^{\phi})$, we have $\varphi(\mathcal{L}_{I}(\mathcal{H}))=\mathcal{L}_{I}(\varphi.\mathcal{H})$.\\

We would now very much like to `realise' those tilted hearts in a functorial way; this is the aim of the following section.  

\subsection{Derived equivalences and tilted hearts}
Even if two hearts of a derived category are in general not derived equivalent, we will show that for $A$ a finite-dimensional symmetric $k$-algebra, and for $\mathcal{T}$ a torsion theory generated by simple objects in ${A}\text{-mod}$, the natural heart ${A}\text{-mod}$ of $\mathcal{D}^{b}(A)$ is derived equivalent to $\mathcal{L}_{\mathcal{T}}(A\text{-mod})$.
\\To obtain such a result, one could naturally relies on the powerful Rickard's theorem on derived equivalences for symmetric algebras (cf. [Ric02, Theorem 5.1]). However, we try here not to use this `bazooka' result, whose underlying machinery is way too sophisticated for what we really need. Instead we will use the characterization of the image of simple modules under an elementary perverse equivalence provided by the Lemma 1.12. Also following [Wo10], we can determine the simple objects in $\mathcal{L}_{S}(A\text{-mod})$ for $S$ a simple $A$-module and this shall reinforce our intuition.
\\Let $T\subset \Irr(A)$ a subset of the the set of isoclasses of simple $A$-modules. We denote by $\mathcal{T}$ the torsion theory generated by $T$. The following definition clarifies the relation between the torsion theory  $\mathcal{T}$ and the derived equivalence $L_{T}$.  From now on, we will denote by $\mathcal{A}:=$${A}$\text{-mod}. 

\begin{definition}
Let $A$ and $B$ be finite dimensional symmetric $k$-algebras and let $\mathcal{T}$ be a torsion theory in $\mathcal{A}$. We say that $F:\mathcal{D}^{b}(A\text{-mod})\stackrel{\sim}{\rightarrow} \mathcal{D}^{b}(B\text{-mod})$ realises $\mathcal{T}$ if $F(\mathcal{L}_{\mathcal{T}}(A\text{-mod}))=B\text{-mod}$. 
\end{definition}

Here, let us stress again that one really wants an equality of hearts and not just an equivalence. We are now ready to establish the connection between torsion theories generated by simple modules and elementary perverse equivalences.

\begin{proposition}
The shifted elementary perverse equivalence $L_{A,T}[1]$ realises the torsion theory $\mathcal{T}$ in $\mathcal{A}$.
\end{proposition}

\begin{proof}
Recall that the elementary perverse equivalence associated to $T$, $L_{T}:\mathcal{D}^{b}(A)\stackrel{\sim}{\rightarrow} \mathcal{D}^{b}(A')$ is such that for $S'\in  \Irr({A'})$:
$$L_{T}^{-1}(S')[-1]=\begin{cases}
  S[-1] \;\text{for}\: S\in T.
  \\ S^{T} \;\text{for}\: S\notin T.
  \end{cases}$$
We have to see that those are exactly the $\mathcal{L}_{ \mathcal{T}}\mathcal{A}$-simples. Using the description of $\mathcal{L}_{ \mathcal{T}}\mathcal{A}$ and Lemma 1.21, it is straighforward to check they are indeed $\mathcal{L}_{ \mathcal{T}}\mathcal{A}$-simple. Finally as for any heart $\heartsuit$ of a bounded $t$-structure of $\mathcal{D}^{b}(A)$, we have an isomorphism $K_{0}(\mathcal{D}^{b}(A))\simeq K_{0}(\heartsuit)$, we have found all the $\mathcal{L}_{ \mathcal{T}}\mathcal{A}$-simples.
\end{proof}

Hence it makes sense to refer to $L_{i}$ as a \textit{simple} perverse equivalence as it realises the simple torsion theory generated by the simple object $S_{i}$. Note that the composition of simple tilts is realised by the corresponding composition of the shifted (by [1]) simple perverse equivalences.\\
From now on, we will consider again the case where $A_{n}$ is the Brauer algebra corresponding to the tree $T_{n}$ and the $F_{i}$'s are those previously defined. A direct application of the above gives the following result (this is a restatement of Proposition 1.14).

\begin{corollary}
For $i\in  \llbracket 1,n \rrbracket$, the elementary perverse equivalence $F_{i}[1]$ realises the torsion theory $\mathcal{T}_{i}$ defined by $ \llbracket 1,n \rrbracket-\{i\}$.
\end{corollary}


\subsection{Looking for relations}
In [ChRo, Proposition 5.71], it is shown that for $I\subset J\subset S$, we have $\mathcal{L}_{I}\mathcal{L}_{J}(\mathcal{H}^{\phi})=\mathcal{L}_{J}\mathcal{L}_{I}(\mathcal{H}^{\phi})$ for any ordered heart $\mathcal{H}^{\phi}$ in $\mathcal{D}^{b}(A\text{-mod})$.\\
Pursuing this philosophy leads us to relate properties of the algebra $A$ with relations between the elementary perverse equivalences. 
Their results can be restated in a more combinatorial way for Brauer tree algebras as we are about to see.
\\We will need the following notations. For a tree $\Gamma$ and $i,j \in \Gamma_{E}$, we will write $i_{\wedge} j=1$ if $\tau(i)\cap \tau(j)=\emptyset$ (\textit{i.e.} if $i$ and $j$ do not share a vertex). Also if $i,j \in \Gamma_{E}$ share a vertex, \textit{i.e.} $i_{\wedge} j\neq1$, and if $j$ is the successor of $i$ in the cyclic ordering, we will write $i<j$. The following propositions can also be found in [Aih10] for simple tilts.

\begin{proposition}
Let $\Gamma$ be a Brauer tree and $I$ and $J$ subsets of $\Gamma_{E}$. Suppose that for every $i\in I\cap (\Gamma_{E}-J)$ and $j\in J\cap (\Gamma_{E}-I)$ with $i_{\wedge} j\neq1$, there exists $k\not\in (I\cup J)$ such that $j<k<i$. Then $L_{I}L_{J}$ is a perverse equivalence and $\mathcal{L}_{J}\mathcal{L}_{I}(A_{\Gamma}\text{-mod})=\mathcal{L}_{I\cup J}\mathcal{L}_{I\cap J}(A_{\Gamma}\text{-mod}).$
\end{proposition}

\begin{proposition}
Let $\Gamma$ be a Brauer tree and $I$ and $J$ subsets of $\Gamma_{E}$. Suppose there exists an involution $\sigma: \Gamma_{E}\stackrel{\sim}{\rightarrow} \Gamma_{E}$, fixing $(I\cap J)\cup (\Gamma_{E}-(I\cup J))$ and inducing a bijection $I\cap (\Gamma_{E}-J)\stackrel{\sim}{\rightarrow} J\cap (\Gamma_{E}-I)$, such that for $i\in I\cap (\Gamma_{E}-J)$ there is no $k\in \Gamma_{E}$ with $\sigma(i)<k<i$. 
\\Then $L_{J}L_{I}L_{J}=L_{I}L_{J}L_{I}$ is a perverse equivalence. Moreover we have in terms of left tilts, $\mathcal{L}_{J}\mathcal{L}_{I}\mathcal{L}_{J}(A_{\Gamma}\text{-mod})=\sigma\mathcal{L}_{I\cup J}\mathcal{L}_{I\cap J}^{2}(A_{\Gamma}\text{-mod})$.
\end{proposition}

We need here to be careful as the permutation $\sigma$ that relabels simples does not lift to a functor and so this only makes sense for tilted hearts.\\
Using the above two propositions, we can decompose any tilt at a torsion theory generated by simple as a composition of simples tilts. This is a combinatorial version of Lemma 1.21.\\ 
Let $\Gamma$ a Brauer tree with $n$ edges and $\phi:\llbracket 1,n \rrbracket \stackrel{\sim}{\rightarrow} \Gamma_{E}$ a labelling of its edges. This gives a labelling on the simples of $A_{\Gamma}\text{-mod}$.

\begin{proposition}
Let $(\Gamma,\phi)$ be a labelled Brauer tree with $n$ edges and $J\subset \llbracket 1,n \rrbracket$. Then there exists $j_{1},\ldots,j_{k} \in \llbracket 1,n \rrbracket$, $\sigma\in \mathfrak{S}_{n}$ and $i\in \mathbb{Z}$ such that $\mathcal{L}_{J}(A_{\Gamma}\text{-mod}^{\phi})=\sigma\mathcal{L}_{j_{k}}^{\pm}\mathcal{L}_{j_{k-1}}^{\pm}\ldots \mathcal{L}_{j_{1}}^{\pm} [i](A_{\Gamma}\text{-mod}^{\phi})$.
\end{proposition}

\begin{proof}
We naturally proceed by induction on the size of $J$ and we shall use a combination of Proposition 1.26 and Proposition 1.27. The case $|J|=1$ is clear.\\ 
Firstly, let us treat the case $J=|2|$ and $J=\{i,j\}$. If $\phi(i), \phi(j)\in \Gamma_{E}$ with $\phi(i)_{\wedge} \phi(j) =1$, then $\mathcal{L}_{\{i,j\}}(A_{\Gamma}\text{-mod}^{\phi})=\mathcal{L}_{i}\mathcal{L}_{j}(A_{\Gamma}\text{-mod}^{\phi})$. If $\phi(i)$ and $\phi(j)$ share a vertex and there is $k\in \Gamma_{E}$ such that $\phi(j)<k<\phi(i)$ then $\mathcal{L}_{\{i,j\}}(A_{\Gamma}\text{-mod}^{\phi})=\mathcal{L}_{j}\mathcal{L}_{i}(A_{\Gamma}\text{-mod}^{\phi})$. If there is $k\in \Gamma_{E}$ such that $\phi(i)<k<\phi(j)$ then $\mathcal{L}_{\{i,j\}}(A_{\Gamma}\text{-mod}^{\phi})=\mathcal{L}_{i}\mathcal{L}_{j}(A_{\Gamma}\text{-mod}^{\phi})$. Finally if there is no such edge $k$, then we found thanks to the proposition above that $\mathcal{L}_{\{i,j\}}(A_{\Gamma}\text{-mod}^{\phi})=(ij)\mathcal{L}_{i}\mathcal{L}_{j}\mathcal{L}_{i}[2](A_{\Gamma}\text{-mod}^{\phi})$.\\
We can then proceed to the induction step and assume that any left tilt $\mathcal{L}_{J}(A_{\Gamma}\text{-mod}^{\phi})$ with $|J|\leq k$ is a product of simple left tilts. Consider now $\mathcal{L}_{K}(A_{\Gamma}\text{-mod}^{\phi})$ where $K\subset \llbracket 1,n \rrbracket$ and $|K|=k+1$. We write $K=\{j_{1},\ldots,j_{k+1}\}$.
\\We start by supposing that there is no vertex adjacent to all $\phi(j)\in K$ so we can assume $\tau(\phi(j_{1}))\cap \tau(\phi(j_{k+1}))=\emptyset$. We take $I=\{j_{1},\ldots,j_{k}\}$ and $J=\{j_{2},\ldots, j_{k+1}\}$, which according to Proposition $1.26$ leads to $\mathcal{L}_{I}\mathcal{L}_{J}(A_{\Gamma}\text{-mod}^{\phi})=\mathcal{L}_{J}\mathcal{L}_{I}(A_{\Gamma}\text{-mod}^{\phi})=\mathcal{L}_{I\cap J}\mathcal{L}_{K}(A_{\Gamma}\text{-mod}^{\phi})$. Hence $\mathcal{L}_{K}(A_{\Gamma}\text{-mod}^{\phi})=\mathcal{L}^{-1}_{I\cap J}\mathcal{L}_{I}\mathcal{L}_{J}(A_{\Gamma}\text{-mod}^{\phi})$ and we can apply the induction.\\
Now, suppose there is a vertex $s$ adjacent to all $\phi(j)$ for $j\in K$ and that the ordered vertices of $\phi(K)$ are $\phi(j_{1}),\ldots,\phi(j_{k+1})$. If $\phi(K)\varsubsetneq \Gamma_{E}^{s}$, we can assume there is $l\in \Gamma_{E}^{s}-\phi(K)$ with $\phi(j_{k+1})<l<\phi(j_{1})$. That way we can take $I=\{j_{1},\ldots,j_{k}\}$ and $J=\{j_{2},\ldots, j_{k+1}\}$, so that $\mathcal{L}_{J}\mathcal{L}_{I}(A_{\Gamma}\text{-mod}^{\phi})=\mathcal{L}_{K}\mathcal{L}_{I\cap J}(A_{\Gamma}\text{-mod}^{\phi})$. On the other hand if $\phi(K)=\Gamma_{E}^{s}$, taking the same $I$ and $J$ as above there is thanks to Proposition 1.27, $\sigma \in \mathfrak{S}_{n}$ such that $\mathcal{L}_{J}\mathcal{L}_{I}\mathcal{L}_{J}(A_{\Gamma}\text{-mod}^{\phi})=\sigma\mathcal{L}_{K}\mathcal{L}_{I\cap J}^{2}(A_{\Gamma}\text{-mod}^{\phi})$. This proves the induction step, and hence we can conclude that any left tilt $\mathcal{L}_{I}(A_{\Gamma}\text{-mod}^{\phi})$ is a product of simple left tilts.
\end{proof}

Consider now the labelled Brauer tree algebra $A_{3}$ with standard labelling $\phi$ of its simples (cf. Example 2). Recall that we have a right action of $G_{3}=\Free(\mathcal{P}(\llbracket 1,3 \rrbracket))\rtimes \mathfrak{S}_{3}$ on $\Tilt(A_{3}\text{-mod}^{\phi})$. Also, the left action of $\TrPic(A_{3})$ on $\Tilt(A_{3}\text{-mod}^{\phi})$ restricts to an action of $\Sph(A_{3})$. For $\mathcal{K} \in \Tilt(A_{3}\text{-mod}^{\phi})$, we write $\Tilt(A_{3}\text{-mod}^{\phi},\mathcal{K})=\{\mathcal{H}\in \Tilt(A_{3}\text{-mod}^{\phi})|\; \mathcal{H}\simeq \mathcal{K}\}$. The previous two propositions lead us to:

\begin{proposition}
We have $\Tilt(A_{3}\text{-mod}^{\phi}, A_{3}\text{-mod}).G_{3}=\Sph(A_{3}).\Tilt(A_{3}\text{-mod}^{\phi})$.
\end{proposition} 

\begin{proof}
We will proceed by induction for both directions. Consider $F_{v}\in \Sph(A_{3})$ with $v \in \mathfrak{S}_{3}$ such that $l(v)>1$. There exists $s\in \mathfrak{S}_{3}$ a simple reflection such that $v=v's$ with $l(v)=l(v')+l(s)$ so that $F_{v}=F_{v'}F_{s}$. Now according to Corollary 1.25, the heart $F_{s}(A_{3}\text{-mod})$ is a tilted heart at a torsion theory generated by simples, so that $F_{s}(A_{3}\text{-mod})=\mathcal{L}_{I}(A_{3}\text{-mod})$ for some $I\subset \llbracket 1,3 \rrbracket$. Now $F_{v}(A_{3}\text{-mod})=F_{v'}F_{s}(A_{3}\text{-mod})=F_{v'}(\mathcal{L}_{I}(A_{3}\text{-mod}))$. But $F_{v'}(\mathcal{L}_{I}(A_{3}\text{-mod}))=\mathcal{L}_{v'(I)}(F_{v'}(A_{3}\text{-mod}))$. Then by induction on the length of $v$, we have that $F_{v}(A_{3}\text{-mod})=\mathcal{L}_{v'(I)}\mathcal{L}_{J}(A_{3}\text{-mod})$ for some $J\subseteq \llbracket 1,3 \rrbracket$. That is to say that for every $F_{v}\in \Sph(A_{3})$, we have $F_{v}(A_{3}\text{-mod})\in \Tilt(A_{3}\text{-mod}^{\phi}, A_{3}\text{-mod}^{\phi}).G_{3}$.\\
For the other direction, we proceed by induction on the numbers of tilts. Hence, we first start by showing that any minimal composition of simple left tilts from $A_{3}\text{-mod}$ to itself does belong to the right-hand side term. 
\\First, we notice that tilting at the simple $S_{2}$ takes us back to a heart which is equivalent to the Brauer tree of the line and the Proposition 1.26 gives us $F_{1}F_{3}[1](A_{3}\text{-mod}^{\phi})=\mathcal{L}_{2}(A_{3}\text{-mod}^{\phi})$ which gives $\mathcal{L}_{2}(A_{3}\text{-mod}^{\phi}) \in \Sph(A_{3}).\Tilt(A_{3}\text{-mod}^{\phi})$.\\
The permutation $(1 3)$ leaves the tree $T_{3}$ invariant (up to isomorphism of tree) so left tilting at the simple $S_{1}$ or at the simple $S_{3}$ play symmetric roles and it suffices to consider starting with tilting at $S_{1}$. 
So when tilting at $S_{1}$, $\mathcal{L}_{1}(A_{3}\text{-mod}^{\phi})$ is equivalent to the Brauer tree algebra of the star but then tilting at any simple will take us back to a heart which is equivalent to the Brauer tree of the line. So in order to conclude, we only have to prove that the following compositions $\mathcal{L}_{3}\mathcal{L}_{1}(A_{3}\text{-mod}^{\phi})$, $\mathcal{L}_{2}\mathcal{L}_{1}(A_{3}\text{-mod}^{\phi})$ and $\mathcal{L}_{1}\mathcal{L}_{1}(A_{3}\text{-mod}^{\phi})$ belong to the right-hand side term. Now applying Proposition 1.26 and Proposition 1.27 gives us the following relations
$$\mathcal{L}_{3}\mathcal{L}_{1}(A_{3}\text{-mod}^{\phi})=F_{2}[-1](A_{3}\text{-mod}^{\phi}),
$$
$$\mathcal{L}_{2}\mathcal{L}_{1}(A_{3}\text{-mod}^{\phi})=(12)F_{3}\mathcal{L}_{2}^{-1}(A_{3}\text{-mod}^{\phi})=(12)F^{-1}_{1}[-1](A_{3}\text{-mod}^{\phi}),
$$
$$\mathcal{L}_{1}\mathcal{L}_{1}(A_{3}\text{-mod}^{\phi})=(23)F_{2}F_{3}F_{2}(A_{3}\text{-mod}^{\phi}).
$$
Hence we have found as claimed that $\mathcal{L}_{3}\mathcal{L}_{1}(A_{3}\text{-mod}^{\phi})$, $\mathcal{L}_{2}\mathcal{L}_{1}(A_{3}\text{-mod}^{\phi})$ and $\mathcal{L}_{1}\mathcal{L}_{1}(A_{3}\text{-mod}^{\phi})$  all belongs to $\Sph(A_{3}).\Tilt(A_{3}\text{-mod}^{\phi})$.\\
Now let $\mathcal{W}$ be a composition of $n$ simple left tilts such that $\mathcal{W}(A_{3}\text{-mod}^{\phi})$ is equivalent to $(A_{3}\text{-mod}^{\phi})$. We have a decomposition $\mathcal{W}(A_{3}\text{-mod}^{\phi})=\mathcal{W}_{1}\mathcal{W}_{2}(A_{3}\text{-mod}^{\phi})$ where $\mathcal{W}_{1}(A_{3}\text{-mod}^{\phi})$ is a composition of $k$ simple left tilts with $k<n$ and $\mathcal{W}_{2}(A_{3}\text{-mod}^{\phi})$ is (up to the permutation $(1 3)$) either $\mathcal{L}_{2}(A_{3}\text{-mod}^{\phi})$ or $\mathcal{L}_{3}\mathcal{L}_{1}(A_{3}\text{-mod}^{\phi})$, $\mathcal{L}_{2}\mathcal{L}_{1}(A_{3}\text{-mod}^{\phi})$ or $\mathcal{L}_{1}\mathcal{L}_{1}(A_{3}\text{-mod}^{\phi})$.\\
Then there is $\sigma \in \mathfrak{S}_{3}$ such that $\mathcal{W}(A_{3}\text{-mod}^{\phi})=\mathcal{W}_{1}F_{\sigma}(A_{3}\text{-mod}^{\phi})$. But thanks to Lemma 1.21, we have $\mathcal{W}_{1}F_{\sigma}(A_{3}\text{-mod}^{\phi})=F_{\sigma}\mathcal{W}'_{1}(A_{3}\text{-mod}^{\phi})$, where $\mathcal{W}'_{1}(A_{3}\text{-mod}^{\phi})$ is a composition of $k$ simple left tilts. By induction, there is $\tau \in \mathfrak{S}_{3}$ such that $\mathcal{W}(A_{3}\text{-mod}^{\phi})=F_{\tau}(A_{3}\text{-mod}^{\phi})$ and this completes our proof.
\end{proof}

In other words, the set of hearts equivalent to the standard heart obtained by tilts at the standard heart is equal to the set of hearts obtained from the standard heart upon applying the spherical twists.

\subsection{Stability conditions}
The space of stability conditions on a triangulated category was introduced by Bridgeland. This set of stability conditions comes equipped with a natural topology and hence yields a new interesting invariant of triangulated categories. Here we follow closely the presentation of the fundamental article [Bri07].

\begin{definition}
A stability function on an abelian category $\mathcal{A}$ is a group homomorphism $Z: K(\mathcal{A})\rightarrow \mathbb{C}$ such that for $0\neq E$, $Z(E) \in \mathbb{H}=\{r\;\exp(i\pi\varphi)\vert \: r\in \mathbb{R}_{>0},\; 0<\varphi\leq 1\}$. We say that $\varphi$ is the phase of $E$.
\\Also, an object $0\neq E \in \mathcal{A}$ is said to be semistable if every subobject $0\neq A\subset E$ satisfies $\phi(A)\leq \phi(E)$.
\end{definition}

Let $Z: K(\mathcal{A})\rightarrow \mathbb{C}$ be a stability function. A \textit{Harder-Narasimhan filtration}, with respect to $Z$, of an object $0\neq E \in \mathcal{A}$ is a finite chain of subobjects
$$0=E_{0}\subseteq E_{1}\subseteq \ldots E_{n-1}\subseteq E_{n}=E
$$ 
whose factors $F_{j}=E_{j}/E_{j-1}$ are semistable objects of $\mathcal{A}$ with 
$$
\phi(F_{1})> \phi(F_{2})> \ldots >\phi(F_{n}).
$$
The stability function $Z$ is said to have the \textit{Harder-Narasimhan property} if every nonzero object of $\mathcal{A}$ has a Harder-Narasimhan filtration.
 
\begin{definition}
A stability condition $(Z, \mathcal{P})$ on a triangulated category $\mathcal{C}$ is the data of a group homomorphism $Z : K( \mathcal{C})\rightarrow\mathbb{C}$ called the central charge, and full additive subcategories $\mathcal{P}(\phi)\subset \mathcal{C}$ for each $\phi \in \mathbb{R}$, satisfying the following axioms: 
\begin{enumerate}
\item if $c \in \mathcal{P}(\phi)$ then $Z (c) = m(c)\exp(i\pi\phi)$ where $m(c)\in \mathbb{R}_{>0}$.

\item $\mathcal{P}(\phi+1)=\mathcal{P}(\phi)[1]$ for all
 $\phi \in \mathbb{R}$.

\item if $c \in \mathcal{P}(\phi)$ and $c' \in \mathcal{P}(\phi')$ with $\phi> \phi'$, then $\Hom(c,c')=0$.

\item for each object $0\neq c\in \mathcal{C}$ there is a finite collection of triangles
$$\xymatrix{ 0=c_{0}\ar[rr] && c_{1}\ar[dl] \ar[r] & \cdots \ar[r] & c_{n-1}\ar[rr] && c_{n}=c \ar[dl]
\\                   & b_{1} \ar@{-->}[ul] &&                  &                    &      b_{n}\ar@{-->}[ul] }$$

with $b_{i}\in \mathcal{P}(\phi_{i})$ where $\phi^{+}(c):=\phi_{1}>\ldots>\phi_{n}=:\phi^{-}(c)$.
\end{enumerate}
\end{definition} 

The collection of the $(\mathcal{P}(\phi))_{\phi\in \mathbb{R}}$ as in the above definition is called a \textit{slicing}. 

The set of stability conditions on $\mathcal{C}$ is denoted by $\Stab(\mathcal{C})$.\\
For $I\subset \mathbb{R}$, define $\mathcal{P}(I)$ to be the extension-closed subcategory of $\mathcal{C}$ generated by the subcategories $\mathcal{P}(\phi)$ for $\phi \in I$. For instance, the objects of $\mathcal{P}((a,b))$ are $0\neq c\in \mathcal{C}$ which satisfy $a<\phi^{-}(c)\leq \phi^{+}(c)<b$ and the zero object of $\mathcal{C}$.\\
Moreover, a slicing $\mathcal{P}$ of a triangulated category $\mathcal{C}$ defines $t$-structures define for every $\phi \in \mathbb{R}$ as $(\mathcal{P}(>\phi),\mathcal{P}(\leq \phi))$ whose corresponding heart is the subcategory $\mathcal{P}((\phi,\phi+1])$. We will follow Bridgeland's convention by defining the heart of the slicing $\mathcal{P}$ as $\mathcal{P}((0,1])$.
\\Should the reader feel overwhelmed by the above definition, the following proposition (cf. [Bri07]) establishes the relation between $t$-structures and stability conditions. That way, we hope this will set the mind of the anxious reader at rest.

\begin{proposition}
To give a stability condition on a triangulated category $\mathcal{C}$ is equivalent to giving a bounded $t$-structure and a stability function on its heart with the Harder-Narasimhan property.
\end{proposition}

For $\mathcal{A}$ a noetherian and artinian abelian category, a stability function on $\mathcal{A}$ always satisfies the Harder-Narasimhan property (cf. [Bri07, Proposition 2.4]). In his original paper, Bridgeland endows $\Stab^{\fin}(\mathcal{C})$, the set of \textit{locally-finite} stability conditions, with a natural topology. We here need not to worry about this technical local-finiteness conditions, as it is always satisfied for the derived category of a finite-dimensional algebra, which is our only concern here. We now state the main result of [Bri07].

\begin{theorem}
For each connected component $\Sigma\subset \Stab(\mathcal{C})$ there is a linear subspace $V(\Sigma)\subset \Hom_{\mathbb{Z}}(K(\mathcal{C}),\mathbb{C})$, with a well-defined linear topology such that $\mathcal{Z}:\Sigma\rightarrow V(\Sigma)$ which maps a stability condition $(Z,\mathcal{P})$ to its central charge $Z$ is a local homeomorphism.
\end{theorem}

There are two groups naturally acting on $\Stab(\mathcal{C})$: the group of automorphisms $\Aut(\mathcal{C})$ and $\mathbb{C}$.
The group of automorphisms $\Aut(\mathcal{C})$ acts continuously on the space $\Stab(\mathcal{C})$ with $\Phi\in \Aut(\mathcal{C})$ acting by $(Z,\mathcal{P})\mapsto (Z\circ \Phi^{-1},\Phi(\mathcal{P}))$.\\ Also there is a $\mathbb{C}$-action on $\Stab(\mathcal{C})$: for $\lambda \in \mathbb{C}$ we have $\lambda.(Z,\mathcal{P}(\phi))=(e^{i\lambda\pi}Z,\mathcal{P}(\phi+\Re(\lambda)))$. If $\Im(\lambda)=0$, this corresponds to rotating the stability condition by $\Re(\lambda)\pi$. In particular, note that the action of the shift $[1]$ coincides with the action of $1\in \mathbb{C}$, and so we can think of the action of the shift functor as rotating the stability condition by $\pi$. Note that the $\mathbb{C}$-action commutes with the action coming from automorphisms of the triangulated category. Lastly, $i\mathbb{R}\subset \mathbb{C}$ acts by rescaling the central charges.

\subsection{Tiling $\Stab(\mathcal{C})$ by tilting}
There is an obvious partition of the space of stability conditions into $(U(\mathcal{H}))_{\mathcal{H}}$, where $U(\mathcal{H})$ is the subset of stability conditions with fixed heart $\mathcal{H}$. Also if $\mathcal{H}$ is a finite length category and has $n$ simple objects, then the subspace $U(\mathcal{H})$ is homeomorphic to $\mathbb{H}^{n}$ where $\mathbb{H}=\{r\;\exp(i\pi\varphi)\vert \: r\in \mathbb{R}_{>0},\; 0<\varphi\leq 1\}$. Indeed, the Harder-Narasimhan property is satisfied for any stability condition in $U(\mathcal{H})$. If we now fix a labelling $\llbracket 1, n \rrbracket\stackrel{\sim}{\rightarrow} \Irr(\mathcal{H})$ the symmetric group $\mathfrak{S}_{n}$ acts on $U(\mathcal{H})$, for $\sigma \in \mathfrak{S}_{n}$ $(Z,\mathcal{P})\mapsto (Z\circ \sigma^{-1},\mathcal{P})$.\\
The strategy initiated by Bridgeland is to think of $U(\mathcal{H})$ for such hearts $\mathcal{H}$ as tiles covering part of the space of stability conditions.\\
We fix a heart $\mathcal{A}$ in $\mathcal{C}$ of finite length with finitely many simple objects.

\begin{proposition}{[Bri05, Lemma 5.5]} Let $\mathcal{A}$ and $\mathcal{B}$ be two hearts of bounded t-structures of $\mathcal{C}$, and let $s$ be a simple object in $\mathcal{A}$. Let $S$ be the (real) codimension one subspace of $\Stab(\mathcal{C})$ consisting of stability conditions for which the simple $s$ has phase $1$ and all other simples have phases in $(0,1)$. Then  $U(\mathcal{A})\cap \overline{U(\mathcal{B})}=S$ if and only if $\mathcal{B}=\mathcal{L}_{s}\mathcal{A}$.
\end{proposition}  

We say that such a subset $S$ is a \textit{codimensional one wall} of $U(\mathcal{A})$. We can extend this definition as follows. A \textit{k-codimensional wall} of $U(\mathcal{A})$ is a $k$-codimensional subspace $S$ of $\Stab(\mathcal{C})$ such that there exists a heart $\mathcal{B}$ with $U(\mathcal{A})\cap \overline{U(\mathcal{B})}=S$. For instance, if $s_{1},\ldots, s_{k}$ are simple objects in $\mathcal{A}$ and $S_{\{1,\ldots,k\}}$ is the subset for which the simples $s_{1},\ldots,s_{k}$ all have phase $1$ and all other simples have phases in $(0,1)$. Then $U(\mathcal{A})\cap \overline{U(\mathcal{L}_{s_{1},\ldots,s_{k}}\mathcal{A})}=S_{\{1,\ldots,k\}}$.
\\Remember that we have denoted  $\Tilt(\mathcal{A})=\{\mathcal{A}_{i}\vert \; i\in I\}$.

\begin{assumption}
The hearts $\mathcal{A}$ and  $\{\mathcal{A}_{i}\vert \; i\in I\}$ the set of all hearts obtained from it by a finite sequence of simple tilts have finite length and only finitely many indecomposables.
\end{assumption}

Under this assumption, Woolf gave a description of the connected component of $\mathcal{A}$ in $\mathcal{C}$ in [Wo10]. Namely, we have the following result:

\begin{theorem}{[Wo10,Theorem 2.18]}
Under the assumption above, the union $\coprod\limits_{i\in I} U(\mathcal{A}_{i})$ is a connected component of the space of stability conditions.
\end{theorem}
Thankfully the assumption $1.35$ will be automatically satisfied for the derived category of the Brauer tree algebra $A_{n}$. From now on, we will denote by $\Stab^{0}(A)$, the connected component of the canonical heart $\mathcal{A}$. This yields a nice description of all hearts in $\Stab^{0}(A_{n})$.

\begin{proposition}
We have the decomposition $\Stab^{0}(A_{n})=\coprod\limits_{i\in I} U(\mathcal{A}_{i})$ where each $\mathcal{A}_{i}$ is equivalent to ${A_{\Gamma}}\text{-mod}$ for $\Gamma$ a Brauer tree with the same numerical invariant as $T_{n}$. 
\end{proposition}

\begin{proof}
Indeed, the assumption $1.35$ is satisfied as Brauer tree algebras are of finite representation type, \textit{i.e.} they have only finitely many indecomposable representations up to isomorphism. Each $\mathcal{A}_{i}$ is obtained from the canonical heart $\mathcal{A}$ by a finite sequence of simple tilts and we have seen that any simple tilt at $\mathcal{A}$ is derived equivalent to $\mathcal{A}$ through an elementary perverse equivalence. Hence each $\mathcal{A}_{i}$ is derived equivalent to $\mathcal{A}$ and moreover there exists a Brauer tree $\Gamma$ with same numerical invariants such that $\mathcal{A}_{i}\stackrel{\sim}\rightarrow {A_{\Gamma}}\text{-mod}$. Moreover the Harder-Narasimhan property is satisfied for any stability conditions in $\Stab^{0}(A_{n})$ as any heart is a module category over a finite-dimensional algebra, and so they are in particular finite length categories.\end{proof}

\section{Towards a description of $\Stab^{0}(A_{n})$}
\subsection{The original hope}
Works of Thomas ([Tho06]), Bridgeland ([Bri09]) and Brav-Thomas ([BrTh11]) give a full description of a connected component of the space of stability conditions of a Kleinian singularity in many cases. Indeed, they show it is the universal cover of some nice space. More precisely, Thomas studied stability conditions of $\mathcal{D}_{n}$, the perfect derived category of $\widehat{A}_{n}$ and showed the following: 

\begin{theorem}
There is a connected component of the space of stability conditions $\Stab(\mathcal{D}_{n})$ which is the universal cover of the configuration space $C_{n+1}$. Moreover, the deck transformations are given by the braid group action of Theorem 1.9.
\end{theorem}

With all that being said, one might hope to obtain similar results for the space of stability conditions of the (classical) derived category of $A_{n}\text{-mod}$.\\
Firstly, we will study the braid group action on the space of stability conditions $\Stab^{0}(A_{n})$, then we will compute the fundamental group of $\Stab^{0}(A_{3})$, and last but not least we will find determine the image of $\mathcal{Z}$ in the case $n=3$.

\subsection{Braid group action on $\Stab^{0}(A_{n})$}

Recall that the subgroup of $\Aut(\mathcal{D}^{b}(A_{n}))$ generated by $F_{1},\ldots, F_{n}$ gives an action of $\Br_{n+1}$ on $\Stab(A_{n})$. This yields at the level of Grothendieck group an action of the symmetric group $\mathfrak{S}_{n+1}$ on $K_{0}(A_{n})^{*}\otimes_{\mathbb{Z}}\mathbb{C}$ through the Bridgeland local homeomorphism. It is natural to wonder how nice this action is. Firstly, we shall make explicit the action of this braid group on the connected component $\Stab^{0}(A_{n})$. We will then discuss its faithfulness before investigating the transitivity of the action.

\begin{proposition}
The Artin braid group $\Br_{n+1}$ acts faithfully on $\{U(\mathcal{H}),\; \mathcal{H}\in \Tilt(A_{n}\text{-mod})\}$.
\end{proposition}

\begin{proof}
We make the braid group $\Br_{n+1}$ act on $\Stab^{0}(A_{n})$ through the assignment $s_{i}\mapsto F_{i}$. First of all, it is clear that there is an action on the set of tiles $\{U(\mathcal{H})$ for $\mathcal{H}\in \Tilt(A_{n})\}$. Now, according to Proposition 1.14 we have that $F_{i}[1](A_{n})$ is the tilted heart at the torsion theory defined by ${T}_{i}=\llbracket 1,n \rrbracket -\{i\}$. We can then apply Lemma 1.21 to decompose this torsion theory as a composition of simple left tilts so that $U(F_{i}[1](A_{n}))\in \Stab^{0}(A_{n})$. But now the connected component $\Stab^{0}(A_{n})$ is invariant under shift. Indeed, we have a continuous action of $\mathbb{R}$ on $\Stab(A_{n})$ that restricts to $\Stab^{0}(A_{n})$ and the action of $1\in \mathbb{R}$ coincides with the shift automorphism. Now any other heart $\mathcal{H}$ in $\Stab^{0}(A_{n})$ is obtained as a composition of simple tilts of $A_{n}\text{-mod}$, that is $\mathcal{H}=\mathcal{L}^{\pm}_{S_{k}}\mathcal{L}^{\pm}_{S_{k-1}}\ldots \mathcal{L}^{\pm}_{S_{1}}(A_{n}\text{-mod})$ where $S_{1}$ is a simple object in $A_{n}\text{-mod}$ and $S_{i}$ is a simple object in $\mathcal{L}^{\pm}_{S_{i-1}}\ldots\mathcal{L}^{\pm}_{S_{1}}(A_{n}\text{-mod})$ for every $i\in \llbracket 2,k \rrbracket$. Then for $\sigma \in \mathfrak{S}_{n+1}$, we have $F_{\sigma}(\mathcal{H})=F_{\sigma}\mathcal{L}^{\pm}_{S_{k}}\mathcal{L}^{\pm}_{S_{k-1}}\ldots \mathcal{L}^{\pm}_{S_{1}}(A_{n}\text{-mod})) =\mathcal{L}^{\pm}_{F_{\sigma}(S_{k})}\mathcal{L}^{\pm}_{F_{\sigma}(S_{k-1})}\ldots\mathcal{L}^{\pm}_{F_{\sigma}(S_{1})}(F_{\sigma}(A_{n}\text{-mod}))$. Hence, we have a braid group action on $\Stab^{0}(A_{n})$. 
\\As for the faithfulness, this turns out to be only a restatement of the (proof of) Proposition 1.1. Indeed, any heart $\mathcal{H}$ in $\Stab^{0}(A_{n})$ can be reached by a sequence of left tilts, and each of these are realisable as a standard equivalence. Moreover, Proposition $1.37$ gives us that $\mathcal{H}=B\text{-mod}$ for $B$ a Brauer tree algebra and consequently we have that $\TrPic(A_{n})\simeq\TrPic(B)$. Now the natural map $\Out(B) \rightarrow \StPic(B)$  is injective . This is true since if $\varphi$ is a non-inner automorphism of $B$, then the twisted bimodule $B_{\varphi}$ is non-isomorphic to $B$ in the stable category of bimodules. On the other hand, the image of $\Sph(A_{n})$ in $\StPic(A_{n})$ is trivial since each generator induces the identity automorphism on $A_{n}\text{-stmod}$. In the end, we see that $\Sph(A_{n})\cap \Out(B)=1$. In other words, the heart $\mathcal{H}$ has trivial stabiliser under the action of $\Sph(A_{n})$ on the set of hearts. But now by faithfulness of the action of $\Sph(A_{n})$ on $\mathcal{D}^{b}(A_{n})$ (cf. Theorem $1.3$), if there is $\sigma \in \mathfrak{S}_{n+1}$ such that $F_{\sigma}(\mathcal{H})=\mathcal{H}$ then $F_{\sigma}=\Id$ and hence $\sigma=1$.
\end{proof}

Note that we could also write the tiling given in Proposition 1.37 as indexed over the isoclasses of planar trees $\Gamma$ with $n$ edges. Indeed, let us denote $V_{\Gamma}=\{U(\mathcal{H}),\; \mathcal{H}\in \Tilt(A_{n}\text{-mod},A_{\Gamma}\text{-mod})\}$ so that we have the decomposition $\Stab^{0}(A_{n})=\bigcup\limits_{\text{isoclasses of trees}\;\Gamma}V_{\Gamma}$.\\
If we denote by $\tau$ the functor induced by the diagram automorphism of $A_{n}$-mod of order two, we have from [RoZi, Theorem 4.5] that $F_{\omega_{0}}[-n]=\tau$ and so to express the $V_{\Gamma}$'s as $G$-torsor for some group $G$ we need to take $G$ a subgroup of $\langle \Br_{n+1},[1] \rangle$ not containing $\tau$. If $n$ is odd we can take $G=\langle \Br_{n+1},[2]\rangle$ for then $[1]\not\in G$, $\tau \not\in G$ but $\tau[1]\in G$.

\begin{proposition}
For $n$ odd, the group $\langle \Br_{n+1},[2]\rangle$ acts faithfully on $\{U(\mathcal{H}),\; \mathcal{H}\in \Tilt(A_{n}\text{-mod})\}$.
\end{proposition}

\begin{proof}
This is again a consequence of [RoZi03, Proposition 4.4] from which we have that the image of $\langle \Br_{n+1},[2]\rangle$ in $\Out(A_{n})$ is also trivial and we can adapt the proof of the previous proposition.
\end{proof}

The next natural question is whether the braid group action defined above is transitive on each subset $V_{\Gamma}$ of $\Stab^{0}(A_{n})$. Firstly, one might wants to extend the group action by adjoining the action of the shift functor, \textit{i.e.} to $\langle \Br_{n+1},[1] \rangle$.

Next, we determine what happens for our main case of interest, the Brauer tree algebra $A_{3}$.
\begin{proposition}
The extended braid group $\langle{\Br_{4}},[1]\rangle$ acts  transitively on the sets $\{U(\mathcal{H}),\; \mathcal{H}\in \Tilt(A_{3}\text{-mod})\;\text{such that}\; \mathcal{H}\simeq A_{3}\text{-mod} \}$ and  $\{U(\mathcal{H}),\; \mathcal{H}\in \Tilt(A_{3}\text{-mod})\;\text{such that}\; \mathcal{H}\simeq A_{\St_{3}}\text{-mod} \}$ where $\St_{3}=\mathcal{L}_{1}(A_{3}\text{-mod})$.
\end{proposition}

\begin{proof}
Thanks to Proposition 1.29, the first assertion comes for free. As for the second assertion, we use Lemma 1.21 as follows. Let $\mathcal{K}$ be a heart such that $\mathcal{K}\simeq \St_{3}$, so that there is a heart $\mathcal{H}\simeq A_{3}$ and a simple object $T$ of $\mathcal{H}$ such that $\mathcal{L}_{T}\mathcal{H}=\mathcal{K}$. There now exists $g\in <{\Br_{4}},[1]>$ such that $F_{g}(A_{3}\text{-mod})=\mathcal{H}$ and according to Lemma 1.21, we have $F_{g}(\mathcal{L}_{F_{g}^{-1}(T)}A_{3}\text{-mod})=\mathcal{L}_{T}F_{g}(A_{3}\text{-mod})$, hence we deduce that $\mathcal{K}=F_{g}(\mathcal{L}_{F_{g}^{-1}(T)}A_{3}\text{-mod})$.\\
We cannot have $F_{g}^{-1}(T)=S_{2}$ as $\mathcal{L}_{S_{2}}(A_{3}\text{-mod})\simeq A_{3}\text{-mod}$ and then $\mathcal{K}\simeq F_{g}(A_{3}\text{-mod})$ which is clearly a contradiction. Consequently we either have that $F_{g}^{-1}(T)=S_{1}$ or $S_{3}$. Assume first that $F_{g}^{-1}(T)=S_{1}$, we then obtain $\mathcal{K}=F_{g}(\mathcal{L}_{1}A_{3}\text{-mod})=F_{g}(A_{\St_{3}}\text{-mod})$. We are now left with the last case where $F_{g}^{-1}(T)=S_{3}$. Consider the heart $\mathcal{H}'=\mathcal{R}_{1}\mathcal{L}_{3}(A_{3}\text{-mod})$, then $\mathcal{H}'\simeq A_{3}\text{-mod}$ and so $\mathcal{L}_{3}(A_{3}\text{-mod})=\mathcal{L}_{1}\mathcal{H}'$. Hence there exists $h\in <{\Br_{4}},[1]>$, such that $F_{h}(A_{3}\text{-mod})=\mathcal{H}'$. This leads to $\mathcal{L}_{3}(A_{3}\text{-mod})=F_{h}\mathcal{L}_{F_{h}^{-1}(S_{1})}(A_{3}\text{-mod})$. Now we cannot have $F_{h}^{-1}(S_{1})=S_{2}$ for the same reason as above. If $F_{h}^{-1}(S_{1})=S_{3}$,  this implies that $\mathcal{L}_{3}(A_{3}\text{-mod})=F_{h}\mathcal{L}_{3}(A_{3}\text{-mod})$ and by faithfulness of the action this gives $h=1$, so we found that  $\mathcal{L}_{3}A_{3}\text{-mod}=\mathcal{L}_{1}(A_{3}\text{-mod})$ which is a contradiction. In the end, we found that $F_{h}^{-1}(S_{1})=S_{1}$ and $\mathcal{K}=F_{g}(\mathcal{L}_{1}(A_{3}\text{-mod}))=F_{g}(A_{\St_{3}}\text{-mod})$.
\end{proof}
 
\subsection{An interesting short exact sequence}
We start by recalling a classical result in algebraic topology (cf. [Ha02, Proposition 1.40]) which will be our tool to obtain an exact sequence relating the fundamental group of $\Stab^{0}(A_{n})$ and the braid group $\Br_{n+1}$.
\begin{definition}
A group $G$ acts properly discontinuously on a space $X$ if each $x\in X$ has a neighborhood $U_{x}$ such that all the $gU_{x}$ for $g\in G$ are disjoint, i.e. $gU_{x}\cap hU_{x}\neq \emptyset$ implies $g=h$.
\end{definition}

\begin{proposition}
Let $G$ be a group acting properly discontinuously on a topological space $X$, then
\begin{itemize}
\item The quotient map $p:X\rightarrow X/G$ is a covering map. 
\item If $X$ is path-connected, $G$ is the group of deck transformation of this covering.
\item If moreover $X$ is path-connected and locally path-connected, then the following short sequence is exact:
\begin{eqnarray}
\xymatrix{1\ar[r] & \pi_{1}(X)\ar[r]^{p^{*}} & \pi_{1}(X/G )\ar[r] & G \ar[r] & 1.}
\end{eqnarray}
\end{itemize}
\end{proposition}
We now wish to obtain this short exact sequence for $X=\Stab^{0}(A_{n})$ and a well chosen group $G$. In fact, we will show that the Artin braid group $\Br_{n+1}$ acts properly discontinuously on $\Stab^{0}(A_{n})$. 

\begin{proposition}
The Artin braid group $\Br_{n+1}$ acts properly discontinuously on $\Stab^{0}(A_{n})$. Moreover for $n$ odd, the extended braid group $\langle \Br_{n+1},[2]\rangle$ acts properly discontinuously on $\Stab^{0}(A_{n})$. 
\end{proposition}

\begin{proof}
We will directly go on proving the second part of the statement, as the first part is an immediate consequence of the second stronger one (for $n$ odd). Clearly, we have that $U(\mathcal{H})\cap U(\mathcal{H'})= \emptyset$ for two differents hearts $\mathcal{H}$ and $\mathcal{H'}$. Moreover for our case of interest, we can translate Rouquier-Zimmermann's Proposition $1.1$ into $g.U(\mathcal{H})\cap U(\mathcal{H})=\emptyset$, $\forall g \in \langle \Br_{n+1}, [2]\rangle$ and for any heart $\mathcal{H}\in \Stab^{0}(A_{n})$. This tells us that $\langle \Br_{n+1}, [2]\rangle$ acts faithfully on the set of tiles $\{U(\mathcal{H})\}_{\mathcal{H}}$ (cf. Proposition 2.3), \textit{i.e.} we have the properness of the action. If this can be considered as a global phenomenon on the tiles $U(\mathcal{H})$, we need here a local version of this property. \\
Let $x\in \Stab^{0}(A_{n})$ with heart $\mathcal{H}$. Consider $P_{x}$ the set of phase of objects of $\mathcal{H}$. As $\mathcal{H}$ has finite representation type, then $P_{x}$ has well defined minimum and maximum. We then denote by $m_{x}$ (resp. $M_{x}$) the minimum (resp. the maximum) of phases of stable objects of $\mathcal{H}$ and we have $0<m_{x}\leq M_{x} \leq 1$.
\\Now $x \in \accentset{\circ}{U(\mathcal{H})}$ if and only if $M_{x}<1$. In this case we can find an open neighborhood $V_{x}\subset \accentset{\circ}{U(\mathcal{H})}$, and we can use the global property above to conclude that for any $g\in \langle \Br_{n+1}, [2]\rangle$,  $F_{g}(V_{x})$ and $V_{x}$ are contained in tiles with different hearts.\\
We are left to consider the case where $M_{x}=1$, \textit{i.e.} $x$ lies in the boundary of the tile $U(\mathcal{H})$. By rotating this stability condition by $\frac{-m_{x}}{2}$, we get a new stability condition $x'$ with same heart $\mathcal{H}$ such that  $m_{x'}=\frac{m}{2}$ and $M_{x'}=M_{x}-\frac{m}{2}<1$, so $x'\in \accentset{\circ}{U(\mathcal{H})}$. Finally, remember that we have noticed earlier that the action of phase shifting commutes with the action coming from automorphisms of the triangulated category. Hence we have that the properly discontinuity property translates from $x'$ to $x$.\end{proof}

We have now came to a stage of our work where we have to restrict ourselves to the case of $A_{3}$, the Brauer tree algebra of the line with three edges and no exceptional vertex. Remember that we have labelled the simples $A_{3}$-modules as follows: $$ \xymatrix{\bullet \ar@{-}[r]^{S_{1}} & \bullet \ar@{-}[r]^{S_{2}}&  \bullet \ar@{-}[r]^{S_{3}}& \bullet}$$ 
Let $\Br_{4}$ be the braid group of type $A_{3}$ and let $\widetilde{\Br_{4}}$ be its extension defined as $\widetilde{\Br_{4}}=\langle s_{1},s_{2},s_{3},z\; | s_{i}s_{j}s_{i}=s_{j}s_{i}s_{j}\: \text{for}\: |i-j|=1,\: s_{i}s_{j}=s_{j}s_{i},\: zs_{1}^{-1}z=s_{3},\; zs_{2}^{-1}z=s_{2},\;  zs_{3}^{-1}z=s_{1},\; z^{3}=s_{1}s_{2}s_{3}s_{1}s_{2}s_{1} \rangle$. Then $\widetilde{\Br_{4}}$ acts on $\mathcal{D}^{b}(A_{3})$, where $s_{i}$ acts as $F_{i}$ and $z$ acts as $[1]\circ \tau$ where $\tau$ is the automorphism of $\mathcal{D}^{b}(A_{3})$ induced by the diagram automorphism of $A_{3}$-mod: $\tau(S_{1})=S_{3}$, $\tau(S_{2})=S_{2}$, $\tau(S_{3})=S_{1}$. 

\begin{corollary}
The group $\widetilde{\Br_{4}}$ acts properly discontinuously on $\Stab^{0}(A_{3})$.
\end{corollary}

Hence according to Proposition $2.8$, we have the following result. 

\begin{proposition}
We have the desired short exact sequence
$$\xymatrix{1\ar[r] & \pi_{1}(\Stab^{0}(A_{3}))\ar[r]^{p^{*}} & \pi_{1}(\Stab^{0}(A_{3})/\widetilde{\Br_{4}} )\ar[r] & \widetilde{\Br_{4}}\ar[r] & 1}.$$
\end{proposition}

\subsection{Fundamental group of $\Stab^{0}(A_{3})$}
We will see that $\Stab^{0}(A_{3})$ is simply connected as expected but to our biggest surprise the Bridgeland local homomorphism is not a covering map !\\
To show simply connectedness, we will exhibit a surjective section $\lambda: \widetilde{\Br_{4}} \longrightarrow \pi_{1}(\Stab^{0}(A_{3})/\widetilde{\Br_{4}})$. 

\subsubsection{Definition of $\lambda$}
Firstly, we fix a basepoint $x\in \Stab^{0}(A_{3})$ with heart $A_{3}$-mod. Let $\hat{\alpha_{1}}$ be a path in $\Stab^{0}(A_{3})$ starting at $x$, crossing a wall of codimension 1 via left tilt at $S_{1}$, crossing another wall via left tilt at $ S_{2}'=\left(\begin{matrix}
   S_{1}  \\
   S_{3}
\end{matrix}\right)$ (the new simple in the new heart) and ending at the point $\phi^{-1}_{1}(x)$.
\\Such a path is well-defined up to homotopy, and we are using the fact that $(12)\mathcal{L}_{2}\mathcal{L}_{1}(A_{3}\text{-mod})=\phi_{1}^{-1}(A_{3}\text{-mod})$ (cf. Proposition $1.26$ and Proposition $1.27$). The path $\hat{\alpha_{1}}$ descends to a loop in $\Stab^{0}(A_{3})/\widetilde{\Br_{4}}$ defining an element $\alpha_{1} \in  \pi_{1}(\Stab^{0}(A_{3})/\widetilde{\Br_{4}})$.
\\A path $\hat{\alpha_{3}}$ is defined exactly in the same way, replacing all the one's by three's, it starts at $x$ and ends at $\phi_{3}^{-1}(x)$. This gives a loop $\alpha_{3}\in  \pi_{1}(\Stab^{0}(A_{3})/\widetilde{\Br_{4}})$.\\
Next, define a path $\widehat{\alpha_{2}}'$ in $\Stab^{0}(A_{3})$ starting at $x$, crossing a codimension-one wall by right tilt at $S_{2}$ and ending at $\phi_{2}^{-1}[1]\tau(x)$. Indeed, we have that $\mathcal{R}_{2}(A_{3}\text{-mod})=\mathcal{L}_{\{1,3\}}^{-1}(A_{3}\text{-mod})$ and the left tilt $\mathcal{L}_{\{1,3\}}(A_{3}\text{-mod})$ is precisely realised by $\phi_{2}$ as $\mathcal{L}_{\{1,3\}}(A_{3}\text{-mod})=\phi_{2}(A_{3}\text{-mod})[-1]$. This induces an element $\alpha_{2}'\in  \pi_{1}(\Stab^{0}(A_{3})/\widetilde{\Br_{4}})$.\\
Finally, let $\hat{\xi}$ be the path in $\Stab^{0}(A_{3})$ starting at $x$ and ending at $[-1]\tau(x)$ defined by first proceeding from $x$ to $[-1]x$ by shift of slope, followed by a path from $[-1]x$ to $[-1]\tau(x)$ lying in the tile $U(A_{3}\text{-mod}[-1])$. The well-defined image in $\pi_{1}(\Stab^{0}(A_{3})/\widetilde{\Br_{4}})$ is denoted by $\xi$. We can then define the map $\lambda$ as follows.

\begin{proposition}
 The map $\lambda: \widetilde{\Br_{4}} \longrightarrow \pi_{1}(\Stab^{0}(A_{3})/\widetilde{\Br_{4}})$ defined by $\sigma_{1}\mapsto \alpha_{1}$, $\sigma_{2}\mapsto \alpha_{2}:=\alpha_{2}'\xi$, $\sigma_{3}\mapsto \alpha_{3}$ and $z\mapsto \xi$, is well defined.
 \end{proposition}
 
\subsubsection{Properties of $\lambda$} 

We now present a result that relates contractible loops in $\Stab^{0}(A_{3})$ with relations amongst left tilts of $A_{3}\text{-mod}$. This will be our key tool to study the fundamental group of the space of stability conditions.

\begin{lemma}
\begin{enumerate}
\item Let $\beta$ be the path in $\Stab^{0}(A_{3})$ starting at $x$, crossing the codimension-one wall via left tilt at $S_{1}$, crossing another wall via left tilt at $S_{2}'$ and finally crossing another wall via left tilt at $S_{1}''=S_{2}$, and ends at $y$. Let $\beta'$ be the path in $\Stab^{0}(A_{3})$ starting at $x$, crossing the codimension-two wall by left tilt at the torsion theory generated by $S_{1}$ and $S_{2}$. Then $\beta$ and $\beta'$ end in the same tile and we take the target of $\beta'$ to be $y$. Moreover we have an homotopy of paths $\beta \sim \beta'$.
\item Let $\beta$ be the path in $\Stab^{0}(A_{3})$ starting at $x$, that first crosses the codimension-one wall via left tilt at $S_{1}$, and then an another wall via left tilt at $S_{3}$ and ends at $y$. Let $\beta'$ be the path in $\Stab^{0}(A_{3})$ starting at $x$, crossing the codimension-two wall by left tilt at the torsion theory generated by $S_{1}$ and $S_{3}$. Then $\beta$ and $\beta'$ end in the same tile and we take the target of $\beta'$ to be $y$. Moreover we have an homotopy of paths $\beta \sim \beta'$.

\end{enumerate}
\end{lemma}

\begin{proof}
The $\mathbb{R}$-action on $\Stab^{0}(A_{n})$ being continuous, we have for every $x\in \Stab^{0}(A_{n})$ and $\varepsilon\in \mathbb{R}$ a path
$$\theta_{x,\varepsilon}:
    \begin{array}{l}
[0,1] \rightarrow \Stab^{0}(A_{n})    \\
t\longmapsto (t\varepsilon).x.   
      \end{array}
$$
Note that for instance if $\varepsilon=1$, this is a path from $x$ to $x[1]$. We can think of the path $\theta_{x,\varepsilon}$ as rotating our stability condition $x$ by $\varepsilon\pi$. Now given any path $\gamma$ in $\Stab^{0}(A_{n})$ starting at a stability condition $x$ with heart $A_{n}$-mod and ending at a stability condition $y$. The following bounds a square
$$\xymatrix{x\ar@{~)}[r]^{\theta_{x,\varepsilon}}\ar@{~)}[d]_{\gamma} & \varepsilon.x  \ar@{~)}[d]^{\varepsilon.\gamma}\\
y\ar@{~)}[r]^{\theta_{y,\varepsilon}} & \varepsilon.y }
$$
through the homotopy $$H:
    \begin{array}{l}
[0,1]\times [0,1] \rightarrow \Stab^{0}(A_{n})    \\
(s,t) \longmapsto t.\varepsilon(\gamma(s)).   
      \end{array}
$$
with $H(s,0)=\gamma(s)$, $H(s,1)=\varepsilon.\gamma(s)$ and $H(0,t)=(t\varepsilon).x$, $H(1,t)=(t\varepsilon).y$.\\
Now this allows us to change our paradigm as we can now think of crossing walls as rotations: either we make the central charges of some simples in a stability condition cross the real axis or equivalently, we rotate the stability condition so that those simples cross the real axis. However, as soon as a simple crosses the real axis, which corresponds to tilting at that simple, we then land in the corresponding tilted heart with new simples. 
\\We now apply this for the case of the Brauer tree algebra $A_{3}$.\\
Let $x\in \Stab^{0}(A_{3})$ with standard heart such that $Z_{x}(S_{1})=Z_{x}(S_{2})$ and $\varphi_{x}(S_{3})>\varphi_{x}(S_{1})$ and $y\in \Stab^{0}(A_{3})$ lying in the same tile such that $\varphi_{y}(S_{1})< \varphi_{y}(S_{2})< \varphi_{y}(S_{3})$. As $x$ and $y$ belong to the same tile, there exists a path $\gamma$ joining $x$ to $y$. We can rotate $x$ and $y$ by some $\varepsilon$ so that we make $S_{1}$ and $S_{2}$ cross the real axis in both situation.\\The path $\theta_{x,\varepsilon}$ corresponds to crossing one wall of codimension two corresponding to $\{S_{1},S_{2}\}$. The path $\theta_{y,\varepsilon}$ corresponds to crossing three walls of codimension one: we first make $S_{1}$ cross, then $ S_{2}'=\left(\begin{matrix}
   S_{1}  \\
   S_{2}
\end{matrix}\right)$ and finally $S_{2}$. Indeed, when $Z_{y}(S_{1})$ just crosses the axis, we land in $U(\mathcal{L}_{1}A)$ and $S_{2}'$ is then a simple object with $\varphi_{y}(S_{2})< \varphi_{y}(S'_{2})$, so this needs to cross the real axis before $S_{2}$ does. Now, once $S'_{2}$ crosses the real axis, $S_{2}$ becomes a simple object again.\\
Now, we have the equality of tilted hearts $\mathcal{L}_{\{S_{1},S_{2}\}}(A_{3}\text{-mod})=\mathcal{L}_{S_{2}}\mathcal{L}_{S_{2}'}\mathcal{L}_{S_{1}}(A_{3}\text{-mod})$, so both $\varepsilon.x$ and $\varepsilon.y$ are in the same tile $U(\mathcal{L}_{\{S_{1},S_{2}\}}(A_{3}\text{-mod}))$. This is giving us the first statement of the lemma as we have that the two paths $\theta_{x,\varepsilon}$ and $\theta_{y,\varepsilon}$ are homotopic.
\\As for the second part, one proceeds, of course, similarly and the corresponding algebraic relation is then $\mathcal{L}_{S_{3}}\mathcal{L}_{S_{1}}(A_{3}\text{-mod})=\mathcal{L}_{S_{1},S_{3}}(A_{3}\text{-mod})$.
\end{proof}

\begin{remark}
Although we have stated the previous lemma only for the algebra $A_{3}$, it is straightforward to write an analogous result for any Brauer tree algebra. The disjonction of cases will then correspond to whether two simples share a vertex or not, or equivalently if there are (or not) non-trivial extensions between the simples.
\end{remark}

The underlying slogan of the lemma, that rotating stability conditions corresponds to the composition of tilting at the semi-stable objects that cross the real axis, gives the following major consequences. The proof of the above lemma has very strong algebraic consequences as it will give us relations between tilted hearts. For instance, we can give different ways of expressing the shifted heart $(A_{3}\text{-mod})[1]$ in terms of compositions of left tilts at simples objects.  In what follows, we will abbreviate $\mathcal{L}_{\mathcal{T}}(\mathcal{A}_{3})$ for the tilted heart $\mathcal{L}_{\mathcal{T}}(A_{3}\text{-mod})$.

\begin{lemma}
We have the following relations between tilted hearts, $$(123)\mathcal{L}_{1}\mathcal{L}_{3}\mathcal{L}_{1}\mathcal{L}_{2}\mathcal{L}_{3}(\mathcal{A}_{3})=(1 3 2)\mathcal{L}_{3}\mathcal{L}_{1}\mathcal{L}_{2}\mathcal{L}_{3}\mathcal{L}_{2}(\mathcal{A}_{3})=(13)\mathcal{L}_{2}\mathcal{L}_{\{1,3\}}\mathcal{L}_{2}\mathcal{L}_{\{1,3\}}(\mathcal{A}_{3})=\mathcal{A}_{3}[-1],$$
$$(132 )\mathcal{L}_{3}\mathcal{L}_{1}\mathcal{L}_{3}\mathcal{L}_{2}\mathcal{L}_{3}(\mathcal{A}_{3})= (123) \mathcal{L}_{3}\mathcal{L}_{1}\mathcal{L}_{3}\mathcal{L}_{2}\mathcal{L}_{1}(\mathcal{A}_{3})=\mathcal{L}_{\{1,3\}}\mathcal{L}_{2}\mathcal{L}_{\{1,3\}}\mathcal{L}_{2}(\mathcal{A}_{3})=\mathcal{A}_{3}[-1].$$
\end{lemma}

\begin{proof}
Let $x\in \Stab^{0}(A_{3})$ with standard heart $\mathcal{A}_{3}$ such that $\phi_{x}(S_{3})>\phi_{x}(\left(\begin{matrix}
   S_{3}  \\
   S_{2}
\end{matrix}\right))>\phi_{x}(S_{1})>\phi_{x}(S_{1})$. The semi-stable objects are exactly, ordered by decreasing phase, $S_{3},\left(\begin{matrix}
   S_{3}  \\
   S_{2}
\end{matrix}\right), S_{1},\left(\begin{matrix}
   S_{1}  \\
   S_{2}
\end{matrix}\right),S_{2}$. We want to rotate the stability conditions so that all simples are in the lower half-plane, that is to say that we will then be in the tile $U(\mathcal{A}_{3}[-1])$. Hence, we need to make all the semi-stable objects cross the real axis. By keeping track of the labelling we have that $(123)\mathcal{L}_{1}\mathcal{L}_{3}\mathcal{L}_{1}\mathcal{L}_{2}\mathcal{L}_{3}(\mathcal{A}_{3})=\mathcal{A}_{3}[-1]$.\\
Let now $x\in \Stab^{0}(A_{3})$ with standard heart $\mathcal{A}_{3}$ such that $\phi_{x}(S_{2})>\phi_{x}(S_{3})>\phi_{x}(\left(\begin{matrix}
   S_{2}  \\
   S_{1}
\end{matrix}\right))>\phi_{x}(S_{1})$. The semi-stable objects are, ordered by decreasing phases, $S_{2},\left(\begin{matrix}
   S_{2}  \\
   S_{3}
\end{matrix}\right), S_{3},\left(\begin{matrix}
   S_{2}  \\
   S_{1}
\end{matrix}\right),S_{1}$. This gives us the second relation $(1 3 2)\mathcal{L}_{3}\mathcal{L}_{1}\mathcal{L}_{2}\mathcal{L}_{3}\mathcal{L}_{2}(\mathcal{A}_{3})=\mathcal{A}_{3}[-1]$.\\
Consider a stability condition $x\in \Stab^{0}(A_{3})$ with standard heart $A_{3}$ such that $\phi_{x}(S_{1})=\phi_{x}(S_{3})>\phi_{x}(S_{2})$. So that we have for the semi-stable objects $\phi_{x}(\left(\begin{matrix}
   S_{1} S_{3}  \\
   S_{2}
\end{matrix}\right))>\phi_{x}(\left(\begin{matrix}
   S_{1}  \\
   S_{2}
\end{matrix}\right))=\phi_{x}(\left(\begin{matrix}
   S_{3}  \\
   S_{2}
\end{matrix}\right))$. We then have $(13)\mathcal{L}_{2}\mathcal{L}_{13}\mathcal{L}_{2}\mathcal{L}_{13}(\mathcal{A}_{3})=\mathcal{A}_{3}[-1]$.\\
We proceed similarly for all the other relations.
\end{proof}

We can now come back to our map $\lambda$ and state the following two propositions, whose proofs rely on our two previous results.

\begin{proposition}
The map $\lambda: \widetilde{\Br_{4}} \longrightarrow \pi_{1}(\Stab^{0}(A_{3})/\widetilde{\Br_{4}})$ is a group homomorphism.
\end{proposition}

\begin{proof}
We need to check that the relations coming from $ \widetilde{\Br_{4}}$ are satisfied in $\pi_{1}(\Stab^{0}(A_{3})/\widetilde{\Br_{4}})$. This will boil down to relations amongst the left tilts applied to $A_{3}$-mod as this translates into contractible loops according to the previous lemma. Indeed, for two loops starting at the same heart, we will have to show that they end in the same heart by only using relations coming from the two results above.
\\Let us start by checking that we have the homotopy $\alpha_{1}\alpha_{3}\sim \alpha_{3}\alpha_{1}$. What we need is then to see that the relation $((23)\mathcal{L}_{2} \mathcal{L}_{3})((12) \mathcal{L}_{2} \mathcal{L}_{1})(\mathcal{A}_{3})=((12) \mathcal{L}_{2} \mathcal{L}_{1})((23) \mathcal{L}_{2} \mathcal{L}_{3})(\mathcal{A}_{3})$ can be obtained from the relations of Lemma 2.11 and Lemma 2.13.\\
However we have $$((23)\mathcal{L}_{2} \mathcal{L}_{3})((12) \mathcal{L}_{2} \mathcal{L}_{1})(\mathcal{A}_{3})=(23)(12)\mathcal{L}_{1}\mathcal{L}_{3}\mathcal{L}_{2}\mathcal{L}_{1}(\mathcal{A}_{3})$$
$$=(23)(12)\mathcal{L}_{1}(23)\mathcal{L}_{3}\mathcal{L}_{2}\mathcal{L}_{3}\mathcal{L}_{1}(\mathcal{A}_{3})=(23)(12)(23)\mathcal{L}_{1}\mathcal{L}_{3}\mathcal{L}_{2}\mathcal{L}_{3}\mathcal{L}_{1}(\mathcal{A}_{3}).$$
Similarly for the right hand side, $$((12) \mathcal{L}_{2} \mathcal{L}_{1})((23) \mathcal{L}_{3} \mathcal{L}_{3})(\mathcal{A}_{3})=(12)(23)(12)\mathcal{L}_{3}\mathcal{L}_{1}\mathcal{L}_{2}\mathcal{L}_{1}\mathcal{L}_{3}(\mathcal{A}_{3}).$$
But $\mathcal{L}_{1}\mathcal{L}_{3}(\mathcal{A}_{3})=\mathcal{L}_{3}\mathcal{L}_{1}(\mathcal{A}_{3})$ and hence we have $\alpha_{1}\alpha_{3}\sim \alpha_{3}\alpha_{1}$.\\
We now turn to the braid relation $\alpha_{2}\alpha_{1}\alpha_{2}\sim\alpha_{1}\alpha_{2}\alpha_{1}$. First off, we need to write $\alpha_{2}$ in terms of left tilts but for that we can use that $\mathcal{L}_{2}\mathcal{L}_{13}\mathcal{L}_{2}\mathcal{L}_{13}(\mathcal{A}_{3})=(1 3)\mathcal{A}_{3}[-1]$ (see the lemma above). Hence $\xi$ is given by the composition of left tilts $\mathcal{L}_{2}\mathcal{L}_{13}\mathcal{L}_{2}\mathcal{L}_{13}(\mathcal{A}_{3})$, and as $\xi^{-1}\alpha_{2}$ is given by $(1 3) \mathcal{L}_{1}^{-1}\mathcal{L}_{3}^{-1}(\mathcal{A}_{3})$, we find that the path $\alpha_{2}$ is given by $(13) \mathcal{L}_{2}\mathcal{L}_{13}\mathcal{L}_{2}(\mathcal{A}_{3})$. That way we need to check that the following relation can be obtained only using the relation of Lemma 2.11 and Lemma 2.13.
$$(12)\mathcal{L}_{2}\mathcal{L}_{1}(13)\mathcal{L}_{2}\mathcal{L}_{13}\mathcal{L}_{2}(12)\mathcal{L}_{2}\mathcal{L}_{1}(\mathcal{A}_{3})=(13)\mathcal{L}_{2}\mathcal{L}_{13}\mathcal{L}_{2}(12)\mathcal{L}_{2}\mathcal{L}_{1}(13)\mathcal{L}_{2}\mathcal{L}_{13}\mathcal{L}_{2}(\mathcal{A}_{3})
$$
$$\Longleftrightarrow \mathcal{L}_{1}\mathcal{L}_{3}\mathcal{L}_{1}\mathcal{L}_{23}\mathcal{L}_{1}\mathcal{L}_{2}\mathcal{L}_{1}(\mathcal{A}_{3})=
\mathcal{L}_{3}\mathcal{L}_{12}\mathcal{L}_{3}\mathcal{L}_{2}\mathcal{L}_{3}\mathcal{L}_{2}\mathcal{L}_{13}\mathcal{L}_{2}(\mathcal{A}_{3}).$$
But we can now use the relations of Proposition 1.26 and Proposition 1.27, so that we have on the right hand side
$$\Longleftrightarrow \mathcal{L}_{1}\mathcal{L}_{3}\mathcal{L}_{1}\mathcal{L}_{23}\mathcal{L}_{2}\mathcal{L}_{1}\mathcal{L}_{2}(\mathcal{A}_{3})=
\mathcal{L}_{3}\mathcal{L}_{12}\mathcal{L}_{3}\mathcal{L}_{2}\mathcal{L}_{3}\mathcal{L}_{2}\mathcal{L}_{3}\mathcal{L}_{1}\mathcal{L}_{2}(\mathcal{A}_{3})$$
$$\Longleftrightarrow \mathcal{L}_{1}\mathcal{L}_{3}\mathcal{L}_{1}\mathcal{L}_{23}\mathcal{L}_{2}(\mathcal{B}_{3})=
\mathcal{L}_{3}\mathcal{L}_{12}\mathcal{L}_{3}\mathcal{L}_{2}\mathcal{L}_{3}\mathcal{L}_{2}\mathcal{L}_{3}(\mathcal{B}_{3})$$
where $\mathcal{B}_{3}=\mathcal{L}_{1}\mathcal{L}_{2}(\mathcal{A}_{3})$ is a heart equivalent to the Brauer tree of the star and whose edges have cyclic ordering $1,3,2$. Now the previous relation amounts to
$$\Longleftrightarrow \mathcal{L}_{3}\mathcal{L}_{1}\mathcal{L}_{3}\mathcal{L}_{23}\mathcal{L}_{2}(\mathcal{B}_{3})=
\mathcal{L}_{3}\mathcal{L}_{12}\mathcal{L}_{3}\mathcal{L}_{2}\mathcal{L}_{3}\mathcal{L}_{2}\mathcal{L}_{3}(\mathcal{B}_{3})$$
$$\Longleftrightarrow \mathcal{L}_{1}\mathcal{L}_{3}\mathcal{L}_{23}\mathcal{L}_{2}(\mathcal{B}_{3})=
\mathcal{L}_{12}\mathcal{L}_{3}\mathcal{L}_{2}\mathcal{L}_{3}\mathcal{L}_{2}\mathcal{L}_{3}(\mathcal{B}_{3})$$
$$\Longleftrightarrow \mathcal{L}_{1}\mathcal{L}_{3}\mathcal{L}_{2}\mathcal{L}_{3}\mathcal{L}_{2}(\mathcal{B}_{3})=
\mathcal{L}_{1}\mathcal{L}_{2}\mathcal{L}_{3}\mathcal{L}_{2}\mathcal{L}_{3}\mathcal{L}_{2}\mathcal{L}_{3}(\mathcal{B}_{3})$$
$$\Longleftrightarrow \mathcal{L}_{3}\mathcal{L}_{2}\mathcal{L}_{3}\mathcal{L}_{2}(\mathcal{B}_{3})=
\mathcal{L}_{2}\mathcal{L}_{3}\mathcal{L}_{2}\mathcal{L}_{3}\mathcal{L}_{2}\mathcal{L}_{3}(\mathcal{B}_{3})$$
$$\Longleftrightarrow \mathcal{L}_{3}\mathcal{L}_{2}\mathcal{L}_{3}\mathcal{L}_{2}(\mathcal{B}_{3})=
\mathcal{L}_{2}\mathcal{L}_{3}\mathcal{L}_{2}(23)\mathcal{L}_{3}\mathcal{L}_{2}(\mathcal{B}_{3})$$
(as we have according to Proposition 1.26 and Proposition 1.27, $\mathcal{L}_{3}\mathcal{L}_{2}\mathcal{L}_{3}(\mathcal{B}_{3})=(23)\mathcal{L}_{3}\mathcal{L}_{2}(\mathcal{B}_{3})$)
$$\Longleftrightarrow \mathcal{L}_{3}\mathcal{L}_{2}(\mathcal{B}_{3})=
\mathcal{L}_{2}\mathcal{L}_{3}\mathcal{L}_{2}(23)(\mathcal{B}_{3})$$
$$\Longleftrightarrow \mathcal{L}_{3}\mathcal{L}_{2}(\mathcal{B}_{3})=(23)
\mathcal{L}_{3}\mathcal{L}_{2}\mathcal{L}_{3}(\mathcal{B}_{3})$$
Hence, we have as claimed that $\alpha_{2}\alpha_{1}\alpha_{2}\sim\alpha_{1}\alpha_{2}\alpha_{1}$. Now let us show that $\alpha_{1}\xi\sim \xi \alpha_{3}$. From the second relation of the previous corollary, we see that $\xi$ is given by the composition of left tilts $(12)\mathcal{L}_{3}\mathcal{L}_{1}\mathcal{L}_{2}\mathcal{L}_{3}\mathcal{L}_{2}$ apply to the heart $\mathcal{A}_{3}$ and so we need to check that $$(12)\mathcal{L}_{2}\mathcal{L}_{1}(12)\mathcal{L}_{3}\mathcal{L}_{1}\mathcal{L}_{2}\mathcal{L}_{3}\mathcal{L}_{2}(\mathcal{A}_{3})
=(12)\mathcal{L}_{3}\mathcal{L}_{1}\mathcal{L}_{2}\mathcal{L}_{3}\mathcal{L}_{2}(23)\mathcal{L}_{2}\mathcal{L}_{3}(\mathcal{A}_{3})$$
$$\Longleftrightarrow \mathcal{L}_{1}\mathcal{L}_{2}\mathcal{L}_{3}\mathcal{L}_{1}((23)\mathcal{A}_{3})=(132)\mathcal{L}_{2}\mathcal{L}_{1}\mathcal{L}_{3}\mathcal{L}_{2}((23)\mathcal{A}_{3})
$$
$$\Longleftrightarrow \mathcal{L}_{1}(23)\mathcal{L}_{2}\mathcal{L}_{3}\mathcal{L}_{2}(\mathcal{L}_{1}((23)\mathcal{A}_{3}))=(23)\mathcal{L}_{2}\mathcal{L}_{1}\mathcal{L}_{3}\mathcal{L}_{2}(\mathcal{L}_{1}((23)\mathcal{A}_{3})).
$$
which only boils down to the relation $\mathcal{L}_{1}\mathcal{L}_{2}(23)\mathcal{A}_{3}=\mathcal{L}_{2}\mathcal{L}_{1}(23)\mathcal{A}_{3}$. We know need to verify that we have the relation $\alpha_{2}\xi \sim \xi \alpha_{2}$. This amounts to check that $$
(13)\mathcal{L}_{2}\mathcal{L}_{13}\mathcal{L}_{2}\mathcal{L}_{2}\mathcal{L}_{13}\mathcal{L}_{2}\mathcal{L}_{13}(\mathcal{A}_{3})=\mathcal{L}_{2}\mathcal{L}_{13}\mathcal{L}_{2}\mathcal{L}_{13}(13)\mathcal{L}_{2}\mathcal{L}_{13}\mathcal{L}_{2}(\mathcal{A}_{3})$$
$$\Longleftrightarrow \mathcal{L}_{2}\mathcal{L}_{13}\mathcal{L}_{2}\mathcal{L}_{13}(\mathcal{A}_{3})= \mathcal{L}_{13}\mathcal{L}_{2}\mathcal{L}_{13}\mathcal{L}_{2}(\mathcal{A}_{3}),$$
and this is exactly one of the relation of the Lemma 2.13. Finally we let the reader carefully check as an exercise that $\xi^{3}\sim \alpha_{1}\alpha_{2}\alpha_{3}\alpha_{1}\alpha_{2}\alpha_{1}$.\end{proof}

\begin{proposition}
The map $\lambda: \widetilde{\Br_{4}} \longrightarrow \pi_{1}(\Stab^{0}(A_{3})/\widetilde{\Br_{4}})$ is a surjective homomorphism.
\end{proposition}
\begin{proof}
Take an arbitrary loop in $\Stab^{0}(A_{3})/\widetilde{\Br_{4}}$. It is homotopic to a loop in general position relative to the tiling of $\Stab^{0}(A_{3})$. It follows that one only needs to consider the loops coming from paths in $\Stab^{0}(A_{3})$ starting at $x$ and ending in $L_{2}L_{1}(A_{3}\text{-mod})$, $L_{2}L_{3}(A_{3}\text{-mod})$, $L_{3}L_{1}(A_{3}\text{-mod})$, $L^{-1}_{2}L_{1}(A_{3}\text{-mod})$, $L^{-1}_{2}L_{3}(A_{3}\text{-mod})$, $L_{3}L_{1}^{-1}(A_{3}\text{-mod})$, $L_{2}(A_{3}\text{-mod})$ and show those are in the image of $\lambda$. The strategy is entirely similar to the one used in the proof of the previous lemma (not to say exactly the same) and so we will only show how to argue for the loop $\beta\in \pi_{1}(\Stab^{0}(A_{3})/\widetilde{\Br_{4}})$ arising from the path $x\mapsto L_{2}(x)$. Indeed, the first three loops are direct consequences of the definitions of the loops $\alpha_{1},\alpha_{2},\alpha_{3}$, and for instance to treat the case of the loop $\gamma$ coming from the path starting at $x$ and ending at $L_{3}L_{1}^{-1}(x)$, one should show that $\gamma\alpha_{1} \sim \alpha_{2}$. Back to our loop $\beta$, we know that $\mathcal{L}_{2}(A_{3})=\phi_{3}\phi_{1}[1](A_{3})$, so all we need to show is that $\beta \sim \alpha_{1}^{-1}\alpha_{3}^{-1}\xi$. In other words, we need $$((12)\mathcal{L}_{2}\mathcal{L}_{1})^{-1}((23)\mathcal{L}_{2}\mathcal{L}_{3})^{-1}((12)\mathcal{L}_{3}\mathcal{L}_{1}\mathcal{L}_{2}\mathcal{L}_{3}\mathcal{L}_{2})(\mathcal{A}_{3})=\mathcal{L}_{2}(\mathcal{A}_{3})
$$
$$\Longleftrightarrow \mathcal{L}_{1}^{-1}\mathcal{L}_{2}^{-1}\mathcal{L}_{3}^{-1}\mathcal{L}_{1}^{-1}\mathcal{L}_{1}\mathcal{L}_{3}\mathcal{L}_{2}\mathcal{L}_{1}\mathcal{L}_{2}(\mathcal{A}_{3})=\mathcal{L}_{2}(\mathcal{A}_{3}).
$$
Note that we have here used the second decomposition of the shift [1] given in Lemma $2.13$ when giving $\xi$ as being realised by the composition $(12)\mathcal{L}_{3}\mathcal{L}_{1}\mathcal{L}_{2}\mathcal{L}_{3}\mathcal{L}_{2}(\mathcal{A}_{3})$.\end{proof}

\begin{theorem}
The principal component $\Stab^{0}(A_{3})$ is simply connected.
\end{theorem}

\begin{proof}
Combining the two previous lemmas, we have built a surjective section of the short exact exact sequence of Proposition $2.9$. Hence, this gives us $\pi_{1}(\Stab^{0}(A_{3}))=1$. 
\end{proof}

\subsection{What is not a covering map}
We have said nothing about covering maps of $\Stab^{0}(A_{3})$ so far. Even though in all other known examples (at least to the knowledge of the author), the Bridgeland local homomorphism actually produces a covering map, we will show that this is not the case here. First of all, we need to have a better understanding of the image $\Im(\mathcal{Z})$. We first consider the general situation of the Brauer tree algebra $A_{n}$ before again restricting ourselves to $A_{3}$. Let $P_{n+1}=ker(\Br_{n+1}\rightarrow \mathfrak{S}_{n+1})$ be the pure braid group on $n$ strings.
\\We have the following commutative diagram:
$$\xymatrix{ \Stab^{0}(A_{n}) \ar[rr]^{\mathcal{Z}} \ar[rd]^p && K_{0}(A_{n})^{*}\otimes_{\mathbb{Z}}\mathbb{C} \\ & \Stab^{0}(A_{n})/P_{n+1} \ar[ur]^{\bar{\mathcal{Z}}} }$$
According to Proposition $2.6$ and Proposition $2.9$, we have that $p:\Stab^{0}(A_{n})\rightarrow  \Stab^{0}(A_{n})/P_{n+1}$ is a covering map. Moreover thanks to Theorem $2.16$, this is a universal cover. \\A natural question is now whether the map $\bar{\mathcal{Z}}$ is an isomorphism onto its image, \textit{i.e.} do we have $\Stab^{0}(A_{n})/P_{n+1}\stackrel{\sim}{\rightarrow} \Im(\bar{\mathcal{Z}})$ ?\\
We know that the Bridgeland homeomorphism is locally injective; the question becomes now whether the map $\bar{\mathcal{Z}}$ is injective. We shall see that for $n\geq 3$ this is never injective. Even worse for $n=4$, we will show that there exist stability conditions with non-isomorphic hearts but the same central charge. Then, we will show that $\Stab^{0}(A_{n})\rightarrow \Im(\bar{\mathcal{Z}})$ is not a covering map.

\subsubsection{Combinatorial decomposition maps}
To determine the image $\Im(\mathcal{Z})$, we first `untwist' the action of $\mathfrak{S}_{n+1}$ on the dualized Grothendieck group, so that we have a classical action by permutation of the coordinates.\\
In order to do so, we define a map $d_{A_{n}}$ inspired from the decomposition map of the symmetric group $\mathfrak{S}_{n+1}$ as follows:
 $d_{A_{n}}:\begin{cases}
  \stackrel[i=1]{n+1}{\bigoplus}\mathbb{Z}v_{i} \longrightarrow K_{0}(A_{n})
  \\ v_{1}\longmapsto [S_{1}],
  \\v_{i} \longmapsto (-1)^{i-1}([S_{i-1}]+[S_{i}]),\; \text{for } i\neq 1, n+1,
  \\v_{n+1} \longmapsto (-1)^{n} [S_{n}].
  \end{cases}$ 
  
\begin{remark}  
If we denote the Heller translate functor by $\Omega$, then we have an homological interpretation of the previous map as $d_{A_{n}}(v_{i})=(-1)^{i-1}[\Omega^{i-1}(S_{1})]$. 
\end{remark}

The action of $\mathfrak{S}_{n+1}$ on $K_{0}(A_{n})$ is given by $\phi_{i}$'s, the spherical twists at the projective indecomposable ${A_{n}}$-module $P_{i}$, while $\mathfrak{S}_{n+1}$ acts on $ \stackrel[i=1]{n+1}{\bigoplus}\mathbb{Z}v_{i}$ by permuting coordinates.

\begin{lemma}
The map $d_{A_{n}}$ is a $\mathbb{Z}\mathfrak{S}_{n+1}$-morphism.
\end{lemma}

\begin{proof}
This should only be a matter of careful checking and we will not (even) treat the extreme cases. Recall the action of $\phi_{i}$ on the simple modules: if $j\neq i$, $[S_{j}]$ is left invariant and $\phi_{i}([S_{i}])=-([S_{i-1}]+[S_{i}]+[S_{i+1}])$. We notice that $(-1)^{i-1}([S_{i-1}]-([S_{i-1}]+[S_{i}]+[S_{i+1}]))=(-1)^{i}([S_{i}]+[S_{i+1}])$. Hence for $x\in \mathbb{Z}^{n}$, we have $d_{A_{n}}((i\: i+1).x)=\phi_{i}.d_{A_{n}}(x)$.
\end{proof}

We now dualize $d_{A_{n}}\otimes \mathbb{C}$ and precompose it with Bridgeland's homomorphism so that:
$$  \Stab(\mathcal{D}^{b}(A_{n}))\stackrel{\mathcal{Z}}{\rightarrow} K_{0}(\mathcal{D}^{b}(A_{n}))^{*}\otimes_{\mathbb{Z}}\mathbb{C} \stackrel{d_{A_{n}}^{*}}{\hookrightarrow} \mathbb{C}^{n+1}
$$ and $(d_{A_{n}}^{*}\circ\mathcal{Z})(Z,\mathcal{P})=(z_{1},-(z_{1}+z_{2}),...,(-1)^{n-1}z_{n})$, where $z_{i}=Z([S_{i}])$.\\
\\For a tree $\Gamma$, recall that $\Gamma_{E}$ and $\Gamma_{V}$ denote the set of its edges and of its vertices. More generally, to any tree we can attach a bipartite labelling by assigning a sign $n(v)=\pm1$ to each vertex $v$ so that any two adjacents vertices have opposite signs. There are clearly only two labellings for any given tree. Once we have fixed a labelling, we can define $d_{\Gamma}$ the twisted combinatorial decomposition map of $\Gamma$ as: 
 $\begin{cases}
  \mathbb{Z}^{\Gamma_{V}} \longrightarrow \mathbb{Z}^{\Gamma_{E}} 
    \\ v \longmapsto n(v)\sum\limits_{v\in e}e.
  \end{cases}$ 
The following lemma makes the connection between the $d_{\Gamma}$'s for different trees.

\begin{lemma} Let $\Gamma$ and $\Gamma'$ be two Brauer trees with n edges and no exceptional vertex, and $A_{\Gamma}$, $A_{\Gamma'}$ their associated Brauer tree algebras. Let $F:\mathcal{D}^{b}(A_{\Gamma}) \stackrel{\sim}{\rightarrow} \mathcal{D}^{b}(A_{\Gamma})$ be a composition of elementary perverse equivalences. Then the following diagram commutes up to sign:
$$\xymatrix{ K_{0}(A_{\Gamma})^{*}\otimes_{\mathbb{Z}}\mathbb{C} \ar[rr]^{d_{\Gamma}^{*}} \ar[rd]_{[F]^{*}}^{\sim} && \mathbb{C}^{n+1} \\ &K_{0}(A_{\Gamma'})^{*}\otimes_{\mathbb{Z}}\mathbb{C} \ar[ur]^{d_{\Gamma'}^{*}}}.$$
\end{lemma}

\begin{proof}
We start by proceeding to two reductions. Indeed, we only have to consider the case where $\Gamma$ and $\Gamma'$ are linked by a single elementary perverse equivalence and $\Gamma$ is the line tree $T_{n}$.\\
We label the edges of $T_{n}$ as usual so that we can take $F=L_{i}:\mathcal{D}^{b}(A_{n}) \stackrel{\sim}{\rightarrow} \mathcal{D}^{b}(A_{\Gamma'})$. Also we endow the vertices of $T_{n}$ with the bipartite labelling yielding $d_{A_{n}}$. This provides $\Gamma'$ with an induced labelling  of its edges (as explained earlier) as well as a bipartite labelling of its vertices.
\\At the level of the Grothendieck groups $K_{0}(A_{n})=\mathbb{Z}^{\Gamma_{E}}:=\stackrel[i=1]{n+1}{\bigoplus}\mathbb{Z}e_{i} $, we have $[L_{i}](e_{i\pm1})=-(e_{i}'+e_{\pm 1}')$, $[L_{i}](e_{i})=e_{i}'$ and $[L_{i}](e_{j})=-e_{j}'$ for $j\neq i, i\pm1 $.\\
We have $[L_{i}]\circ d_{A_{n}}(v)=-d_{\Gamma'}(v)$ for each $v\in \Gamma_{V}=\Gamma'_{V}$ and hence:
$$\xymatrix{ \mathbb{C}^{n+1} \ar[rr]^{d_{A_{n}}} \ar[rd]^{d_{\Gamma'}} &&K_{0}(A_{n})^{*}\otimes_{\mathbb{Z}}\mathbb{C}\ar[ld]_{\sim}^{-[L_{i}]} \\ &K_{0}(A_{\Gamma'})^{*}\otimes_{\mathbb{Z}}\mathbb{C}}.$$
We now only have to dualize this diagram to complete the proof. 
\end{proof}

\subsubsection{}Recall that we can identify the Grothendieck group of the triangulated category $\mathcal{D}^{b}(A_{n})$ with the Grothendieck group of any of its bounded hearts. Then, using the previous lemma it becomes easy to compute the image of a stability condition of a given heart in $\mathbb{C}^{n+1}$.\\
For instance for the three-simples case, consider $\sigma$ a stability condition with heart $A_{\St_{3}}$, the star tree with simples $S_{i}$, for $i=1,\ldots,3$. If we denote by $z_{i}:=Z(S_{i})$ for $i=1,\ldots, 3$, then $(d_{A}^{*}\circ\mathcal{Z})(\sigma)=(z_{1},z_{2},-z_{1}-z_{2}-z_{3},z_{3})\in \mathbb{C}^{4}$.\\

\begin{proposition}
For $n\geq 3$, the map $\bar{\mathcal{Z}}:\Stab^{0}(A_{n})/\P_{n+1} \rightarrow \Im(\mathcal{Z})$ is not injective.
\end{proposition}

\begin{proof}
Let $\sigma \in \Stab^{0}(A_{n})$ be a stability condition with standard heart $A_{n}\text{-mod}$ and such that $z_{1}=z_{3}$ (this is indeed possible since $[S_{1}]$ and $[S_{2}]+[S_{3}]$ are linearly independent classes). Let $g=\phi_{1}\phi_{2}\phi_{1}\in \Br_{n+1}\backslash P_{n+1}$, the image of $g$ in $\mathfrak{S}_{n+1}$ is the transposition $(13)$, so the choice of $\sigma$ means that $g(\sigma)$ and $\sigma$ have the same central charge, \textit{i.e.} same image by $\mathcal{Z}$. On the other hand we have $g\not\in P_{n+1}$. Then the two stability conditions are distinct in $\Stab^{0}(A_{n})/\P_{n+1} $.
\end{proof}

Using the combinatorial decomposition maps $d_{\Gamma}$'s, we are able to distinguish, in the three-simples case, stability functions with non-isomorphic hearts. Indeed, in the three-simples case, we only have two trees to deal with, namely the line and the star. But we clearly have $\Im(d_{A_{3}}^{*})\cap \Im(d_{\St_{3}}^{*})=\emptyset$ as in the image of $d_{A_{3}}^{*}$, there are two coordinates lying in the upper half-plane and two in the lower half-plane while for $d_{\St_{3}}^{*}$, there are always three coordinates in the same half-plane. However, the following proposition tells us that this no longer holds for the four-simples case.

\begin{proposition}
There exists two stability conditions $\sigma$ and $\tau$ in $\Stab^{0}(A_{4})$ with non-isomorphic hearts but same central charges.
\end{proposition}

\begin{proof}
We consider the following two trees whose edges are endowed with a complex number such that we have $d_{A_{4}}(3i,i,i,i,i)=(3i,-4i,2i,-2i,i)$ and $d_{\Gamma}(2i,i,i,2i)=(2i,-4i,i,3i,-2i)=(1\;3\;5\;4).(3i,-4i,2i,-2i,i)$:
 $$T_{4}: \xymatrix{\bullet \ar@{-}[r]^{3i} & \bullet \ar@{-}[r]^{i}& \bullet \ar@{-}[r]^{i} &\bullet \ar@{-}[r]^{i}& \bullet}$$
$$ \xymatrix{
 & {\bullet} \ar@{-}[dr]^{2i} \\
\Gamma: & &   {\bullet}  \ar@{-}[r] ^{i}\ar@{-}[dl]_{i}&    {\bullet}  \ar@{-}[r]^{2i}  &   {\bullet} \\
&   {\bullet}} 
  $$
Here we have chosen the labelling of the vertices of the trees so that their far right vertices have minus sign. As the decomposition maps are defined up to sign, another choice of labelling would only lead to a shift by one of those stability conditions.\\
We conclude that in $\Stab^{0}(A_{4})$, we have two stability conditions $(Z,A_{4})$ and $(Z',A_{\Gamma})$ that have the same image under $(d_{A_{4}}^{*}\circ\mathcal{Z})$ and by injectivity of $d_{A_{4}}^{*}$, this says that $\bar{\mathcal{Z}}$ is not injective.
\end{proof}

This proposition should be related to a conjecture of Thomas (cf. [Tho06]). Indeed in the last remark of [Tho06], Thomas conjectured that two stability conditions with the same central charge on a Calabi-Yau category should only differ by an autoequivalence of the category that is the identity at the level of the Grothendieck group. However as we are working over a symmetric algebra, our setting loosely corresponds to a 0-Calabi-Yau category. Hence, this proposition can be thought as providing a limitation to Thomas's conjecture.

\subsubsection{Description of $Im(\mathcal{Z})$}
Inside $\Stab^{0}(A_{3})$, we denote by $\mathcal{U}:=U(A_{3})\sqcup U(\St_{3})$ where the tilted heart $\St_{3}$ is realised from $A_{3}$ by the elementary perverse equivalence $L_{1}$. We denote by $Z_{\mathcal{U}}:=\mathcal{Z}(\mathcal{U})$. Even though $\mathcal{U}$ is not quite a fundamental domain, thanks to Proposition $2.6$ we have that the orbit of $\mathcal{U}$ by the extended braid group action $\langle\Br_{4},[1]\rangle$ will be all of $\Stab^{0}(A_{3})$. Down to the Grothendieck group level, this yields to an action of $\tilde{\mathfrak{S}}_{4}=\langle\mathfrak{S}_{4}, \gamma \rangle$ where $\gamma$ acts by multiplication by $-1$, on $K_{0}(A_{3})^{*}\otimes_{\mathbb{Z}}\mathbb{C}$. Hence as the injectivity of $d_{A}^{*}$ gives $\tilde{\mathfrak{S}}_{4}.Z_{\mathcal{U}}=(d_{A}^{*})^{-1}(\tilde{\mathfrak{S}}_{4}.d_{A}^{*}(Z_{\mathcal{U}}))$, we have that $\Im(\mathcal{Z})=(d_{A}^{*})^{-1}(\tilde{\mathfrak{S}}_{4}.d_{A}^{*}(Z_{\mathcal{U}}))$. The following lemma provides us with an explicit description of $\tilde{\mathfrak{S}}_{4}.d_{A}^{*}(Z_{\mathcal{U}})$ and hence of $\Im(\mathcal{Z})$.

\begin{lemma}
We have $\tilde{\mathfrak{S}}_{4}.d_{A}^{*}(Z_{\mathcal{U}})=\Im(d_{A}^{*})\backslash ((\epsilon_{i}=0, i=1\ldots 4)\sqcup \bigsqcup\limits_{\sigma \in \tilde{\mathfrak{S}}_{4}} \sigma.l)$, where $l$ is the line $\epsilon_{1}=\epsilon_{3}=-\epsilon_{2}=-\epsilon_{4}$.
\end{lemma}

\begin{proof}
Firstly, we write down the image of $U(A_{3})$ and $U(\St_{3})$. We have 
\begin{eqnarray}
d_{A}^{*}\circ \mathcal{Z}(U(A_{3}))=\{(z_{1},-(z_{1}+z_{2}),z_{2}+z_{3},-z_{3});\: z_{i}\in \mathbb{H} \}
\end{eqnarray} and
\begin{eqnarray}
d_{A}^{*}\circ \mathcal{Z}(U(\St_{3}))=\{(z_{1},z_{2},-(z_{1}+z_{2}+z_{3}),z_{3});\: z_{i}\in \mathbb{H} \}.
\end{eqnarray} From that, we easily deduce the inclusion $(\subseteq)$.\\
The image $\Im(d_{A}^{*})$ consists of $\{(\epsilon_{1},\ldots,\epsilon_{4})\in \mathbb{C}^{4}\vert\;  \sum_{i}\epsilon_{i}=0\}$. Consider now $(u_{1},u_{2},u_{3},u_{4})$ belonging to the right hand side. If three out of the four coordinates lie in the same upper half-plane, it is in the $\tilde{\mathfrak{S}}_{4}$-orbit of $(3)$. Otherwise, we can assume that $u_{1},u_{3}\in \mathbb{H}$ and $u_{2},u_{4}\in \mathbb{C}\backslash\mathbb{H}$. Now either $u_{1}+u_{2}<0$ and $(u_{1},u_{2},u_{3},u_{4})\in d_{A}^{*}\circ \mathcal{Z}(U(A_{3}))$ or $u_{1}+u_{2}>0$ and $(u_{1},u_{2},u_{3},u_{4})\in (1\:3)(2\:4).d_{A}^{*}\circ \mathcal{Z}(U(A_{3}))$. This gives us the other inclusion and hence concludes our proof.
\end{proof}

We can then deduce from this lemma a description of the image of $\mathcal{Z}$, where $z_{i}=\mathcal{Z}([S_{i}])$ for $S_{1},S_{2},S_{3}$ the simple $A_{3}$-modules.
 
\begin{corollary}
We have $\Im(\mathcal{Z})=\mathbb{C}^{3}\backslash ((z_{1}=0) \cup (z_{3}=0) \cup (z_{1}+z_{2}=0) \cup (z_{2}+z_{3}=0) \cup \bigsqcup\limits_{\sigma \in \tilde{\mathfrak{S}}_{4}} \sigma.l)$, where $l$ is the line $\{(z,0,z)|\; z\in \mathbb{C}\}$.
 \end{corollary}
We now need another algebraic topology result, that would help us compute the fundamental group of $\Im(\mathcal{Z})$, as it will tell us that removing lines does not alter the fundamental group.

\begin{proposition}
Let $M$ be a connected complex manifold and $N$ a proper closed submanifold of $M$ such that the codimension of $N$ in $M$ is at least $2$. For $x\in M\backslash N$, the inclusion $i: M\backslash N \hookrightarrow M$ induces an isomorphism of fundamental groups $i^{*}:\pi_{1}(M\backslash N,x) \stackrel{\sim}{\rightarrow} \pi_{1}(M,x)$.
\end{proposition}
\begin{proof}
Let $\gamma:[0,1]\rightarrow M$ a loop in $M$ with base point $x$. As $N$ is of real codimension greater than two and the interval $[0,1]$ of real dimension one, we can perturb $\gamma$ into a new loop $\gamma'$ with same base point so that $\Im(\gamma')\cap N=\emptyset$. But since $[\gamma]=[\gamma']$ in $\pi_{1}(M,x)$ and $\gamma'$ is a loop in $M\backslash N$, so that $i^{*}$ is a surjective homomorphism.\\
Consider now $\gamma:[0,1]\rightarrow M\backslash N$ with $g\in \ker(i^{*})$ so that there exists a homotopy $H:[0,1]^{2}\rightarrow M$ from $H\vert_{[0,1]\times \{0\}}=\gamma$ and $H\vert_{[0,1]\times \{1\}}=0_{x}$, the constant loop at $x$. But since $N$ is of real codimension at least four by assumption while $[0,1]^{2}$ is of real dimension two. Just as before, we can perturb $H$ into a new homotopy  $H':[0,1]^{2}\rightarrow M$ from $H'\vert_{[0,1]\times \{0\}}=\gamma$ and $H'\vert_{[0,1]\times \{1\}}=0_{x}$ such that $\Im(H')\cap N=\emptyset$, so that $H'$ is now a homotopy in $M\backslash N$ \textit{i.e.} $i^{*}$ is injective.
\end{proof}

\begin{proposition}
Let $x\in \Im(\mathcal{Z})$, we have that $\pi_{1}(\Im(\mathcal{Z}),x)=\mathbb{Z}^{4}$.
\end{proposition}
\begin{proof}
Thanks to the previous lemma, we now have a better chance of computing $\pi_{1}(\Im(\mathcal{Z}),x)$ for any $x\in \Im(\mathcal{Z})$. Indeed, when computing the fundamental group of $\Im(\mathcal{Z})$, the lines removed do not change the fundamental group according to the previous proposition. Hence we are now left to consider the space $X_{3}:=\mathbb{C}^{3}\backslash ((z_{1}=0) \cup (z_{3}=0) \cup (z_{1}+z_{2}=0) \cup (z_{2}+z_{3}=0) $ where  $\pi_{1}(\Im(\mathcal{Z}))=\pi_{1}(X_{3})$. We have a $\mathbb{C}^{\times}$-bundle $X_{3} \rightarrow \widetilde{X_{3}}, (z_{1},z_{2},z_{3})\mapsto (z_{2}z_{1}^{-1},z_{3}z_{1}^{-1})$ where $\widetilde{X_{3}}:=\mathbb{C}^{2}\backslash ((x=0)\sqcup (y+1=0) \sqcup (x+y=0))$. Henceforth, it only remains to compute the fundamental group of $\widetilde{X_{3}}$. Using [OrTe, Theorem 5.57], we found that $\pi_{1}(\widetilde{X_{3}})=\mathbb{Z}^{3}$ and hence we can conclude that $\pi_{1}(\Im(\mathcal{Z}),x)=\mathbb{Z}^{4}$.
\end{proof}

The collection of all the previous results leads to the following rather surprising result.
\begin{theorem}
The Bridgeland local homeomorphism $\mathcal{Z}:\Stab^{0}(A_{3})\rightarrow \Im(\mathcal{Z})$ is \textbf{not} a covering map.
\end{theorem}
\begin{proof}
We let $\sigma\in \Stab^{0}(A_{3})$ a stability condition with standard heart and let $Z([S_{2}])$ tend to $0$. This is possible in $\Im(\mathcal{Z})$ as the complex hyperplane $Z([S_{2}])=-(\varepsilon_{1}+\varepsilon_{2})=0$ has not been excised in $\Im(\mathcal{Z})$ according to Lemma $2.22$. On the other hand, this degeneration is not permitted in $\Stab^{0}(A_{3})$ as by definition, the central charge of any semistable module (in particular any simple module) cannot tend to $0$. The key point is that in some hearts $Z([S_{2}])$ corresponds to a simple so it cannot tend to 0, while in other hearts there is no such restriction.
\end{proof}

\section{About co-stability conditions}
Some triangulated categories have empty space of stability conditions, like the compact derived category of the dual numbers over an algebraically closed field. This motivated J\o rgensen and Pauksztello to introduce in [JoPa11] the space of co-stability conditions attached to a triangulated category as a `continuous' version of co-$t$-structures. They show that a set of \textit{nice} co-stability conditions is, just like for stability conditions, a complex manifold. By \textit{nice} is meant co-stability conditions satisfying the so-called \textit{condition (S)}, which will turn out to be quite a restriction in the case of symmetric algebras as we are about to see.\\
We do not wish here to get into the details of the theory of co-stability conditions, but one can think of co-stability conditions as a dual version of stability conditions, just like going from $t$-structures to co-$t$-structures. We recall the following definition from [JoPa11].

\begin{definition}
A co-slicing $\mathcal{Q}$ of a triangulated category $\mathcal{C}$ satisfies condition (S) if for every $\phi \in \mathbb{R}$ and any two non-isomorphic indecomposables $x,y \in \mathcal{Q}(\phi)$ we have $\Hom_{\mathcal{C}}(x,y)=0$.
\end{definition}

Also, just as for stability conditions, a co-stability condition is given by a co-heart and a co-stability function on its co-heart which has the split Harder-Narasimhan property.\\
As the reader might not be satisfied by the result of Theorem 2.30, one could hope that the space of co-stability conditions of $\Ho({A_{3}}\text{-proj})$ offers a better result. Alas, as we are about to see, it is in some sense even worse behaved, and not only for Brauer tree algebras but for any symmetric algebra.\\
In the following, we denote by $A$ a finite-dimensional symmetric $k$-algebra. We assume that $A$ is not semi-simple.
 The abelian category ${A}$\text{-mod} of $\mathcal{D}^{b}(A)$ is to hearts what the additive category ${A}$\text{-proj} is to co-hearts in $\Ho({A}\text{-proj})$.\\
For all co-stability conditions $(Z,\mathcal{Q})$ with co-heart ${A}$\text{-proj}, each projective indecomposable $A$-module belongs to a slice $\mathcal{Q}(\phi)$ for some $\phi \in (0;1]$. But as $A$ is a symmetric algebra, for any $P,Q$ two $A$-projectives, one have $\Hom_{A}(P,Q)\simeq\Hom_{A}(Q,P)^{*}$. Hence, because of the condition (S), there is no co-stability conditions with co-heart ${A}$\text{-proj}. Motivated by this remark, we obtain the stronger following result.

\begin{theorem}
Let $A$ be a symmetric non semi-simple $k$-algebra. Then the space of co-stability conditions $\CoStab$$(\Ho({A}\text{-proj}))$ is empty.
\end{theorem}
\begin{proof}
Everything relies on the fact that as $A$ is a symmetric $k$-algebra, the triangulated category $\Ho({A}\text{-proj})$ is a $0$-Calabi-Yau category, \textit{i.e.} for $P^{\bullet}, Q^{\bullet}\in \Ho({A}\text{-proj})$, the Hom-spaces $\Hom_{\Ho({A}\text{-proj})}(P^{\bullet},Q^{\bullet})$ and  $\Hom_{\Ho({A}\text{-proj})}(Q^{\bullet},P^{\bullet})$ are naturally dual (cf. [Ric96, Corollary 3.2]). Now for any co-slicing $\mathcal{Q}$, consider any two indecomposables objects in the co-heart of $\mathcal{Q}$ with non-zero Hom-space between them. By the definition of a co-slicing and as $\Ho({A}\text{-proj})$ is a $0$-Calabi-Yau category, they have to belong to the same slice, and so this co-stability condition cannot satisfy condition (S). Hence, we have that the space of co-stability conditions of $\Ho({A}\text{-proj})$ is empty.
\end{proof}

\end{document}